\documentclass{article}

\usepackage{microtype}
\usepackage{graphicx,epstopdf}
\usepackage{subfigure}
\usepackage{booktabs} 
 
\usepackage{soul}
\usepackage[normalem]{ulem}
\usepackage[prependcaption,textsize=tiny]{todonotes}

\RequirePackage{amsmath}
\RequirePackage{amssymb}
\RequirePackage{amsthm}
\usepackage[T1]{fontenc}
\usepackage{enumerate}

\newcommand{\RR}{\mathbb{R}}
\def\beq{\begin{equation} }\def\eq{\end{equation} }

\def\1{\mathbf{1}}
\newcommand{\argmin}{\mathop{\mathrm{argmin}}}
  
\newcommand{\abs}[1]{\left | #1\right |}

\newcommand{\norm}[1]{\left\Vert#1\right\Vert}

\newcommand{\e}{ \operatorname{\mathbb E} }

 \newcommand{\std}{\operatorname{std} } 
 
 \newcommand{\eps}{\epsilon}

\ifx\theorem\undefined
\newtheorem{theorem}{Theorem}
\fi
\ifx\assumption\undefined
\newtheorem{assumption}[theorem]{Assumption}
\fi
\ifx\proposition\undefined
\newtheorem{proposition}[theorem]{Proposition}
\fi
\ifx\remark\undefined
\newtheorem{remark}{Remark}
\fi
\ifx\definition\undefined
\newtheorem{definition}[theorem]{Definition}
\fi

\ifx\corollary\undefined
\newtheorem{corollary}[theorem]{Corollary}
\fi

\usepackage{hyperref}

\usepackage[accepted]{icml2020}

 \icmltitlerunning{Continuous Model of Stochastic ADMM}

\begin{document}

\twocolumn[
\icmltitle{
Stochastic Modified Equations for Continuous Limit of Stochastic ADMM
}



\icmlsetsymbol{equal}{*}

\begin{icmlauthorlist}
\icmlauthor{Xiang Zhou}{cityu}
\icmlauthor{Huizhuo Yuan}{pku}
\icmlauthor{Chris Junchi Li}{bc}
\icmlauthor{Qingyun Sun}{st}
\end{icmlauthorlist}

\icmlaffiliation{cityu}{School of Data Science and Department of Mathematics,
City University of Hong Kong, Hong Kong, China}

\icmlaffiliation{pku}{Peking University, China}
\icmlaffiliation{bc}{Department of EECS, UC Berkeley, USA}
\icmlaffiliation{st}{Department of Mathematics
Stanford University}

\icmlcorrespondingauthor{Xiang Zhou}{xizhou@cityu.edu.hk}

\icmlkeywords{Machine Learning, ICML}

\vskip 0.3in
]



\printAffiliationsAndNotice{}  

\begin{abstract}
Stochastic version of alternating direction method of multiplier (ADMM) and its variants (linearized ADMM, gradient-based ADMM) plays key role for modern large scale machine learning problems.
One example is regularized empirical risk minimization problem.
In this work, we put different variants of stochastic ADMM into a unified form, which includes standard, linearized and gradient-based ADMM with relaxation, and study their dynamics via a continuous-time model approach. We adapt the mathematical framework of stochastic modified equation (SME), and show that the dynamics of stochastic ADMM is approximated by a class of stochastic differential equations with small noise parameters in the sense of weak approximation. The continuous-time analysis would uncover important analytical insights into the behaviors of the discrete-time algorithm, which are non-trivial to gain otherwise. For example, we could characterize the fluctuation of the solution paths precisely, and decide optimal stopping time to minimize variance of solution paths.

\end{abstract}

\section{Introduction}

For modern industrial scale machine learning problems with massive amount of data, stochastic first-order methods almost become the default choice. Additionally, the datasets are not only extremely large, but often  stored or even collected in a distributed manner. Stochastic version oflternating direction method of multiplier(ADMM) algorithms are popular approachs to handle this distributed setting, especially for the regularized empirical risk minimization problems.
 
Consider the following stochastic optimization problem:
\beq\label{opt_eq}
\underset{x\in \RR^d}{\text{minimize}} \ V(x):= f(x) + g(Ax),
\eq
where $f(x)=\e_{\xi} \ell (x,\xi)$ with $\ell$ as the loss incurred  on a sample $\xi$, $f:\RR^d\rightarrow\RR\cup \{+\infty\}$, $g:\RR^m\rightarrow\RR \cup \{+\infty\}$, $A\in\RR^{m\times d}$, and both $f$ and $g$ are convex and differentiable. 
 The stochastic version of alternating direction method of multiplier (ADMM) \cite{admm3boyd} is to rewrite \eqref{opt_eq} as a constrained optimization problem 
\beq\label{opt_eq2}
\begin{aligned}
&\underset{x\in \RR^d, z\in \RR^m}{\text{minimize}} &&\ \e_{\xi} f(x,\xi) + g(z)
\\
&\text{subject to} &&\  Ax - z = 0.
\end{aligned}
\eq
Here and  through the rest of the paper,  we start to use the same   $f$ for both the   stochastic instance and  the expectation to ease the notation.
In the batch learning setting, $f(x)$ is approximated by the empirical risk function
$f_{emp}=\frac{1}{N}\sum_{i=1}^N f(x,\xi_i)$. However,
 to minimize $f_{emp}$ with a  large amount of samples,
 the computation is   less efficient under time and resource constraints.
In the stochastic setting, in each iteration $x$ is updated based on one noisy sample  $\xi$
 instead of a full training set.

Note that the classical setting of linear constraint $Ax + Bz = c$ can be reformulated as $z = Ax$ by a simple linear transformation operation when $B$ is invertible.
 
%

 One of the  main ideas in the stochastic ADMM is in parallel  to the stochastic gradient descent (SGD).
 At iteration $k$, 
 an {\it iid} sample $\xi_{k+1}$ is drawn from the distribution of $\xi$.
 A straightforward  application of   this SGD idea to 
the  ADMM for solving  \eqref{opt_eq2} leads to 
the following  {\bf stochastic ADMM (sADMM)}
\begin{subequations} \label{sADMM}
\begin{align}
 \label{sADMM_x}
x_{k + 1} &=\argmin_{x} \left\{ f(x,\xi_{k+1}) + \frac{\rho}{2} \| A x - z_k + u_k
            \|_2^2\right\},
\\
\label{sADMM_z}
z_{k + 1}
&=
\argmin_{z}\bigg\{ g(z) + \frac{\rho}{2} \left\|   \alpha A x_{k+1} +(1-\alpha)z_k
   - z + u_k
        \right\|_2^2\bigg\},
        \\
  \label{sADMM_u}
u_{k + 1}&=u_k +  (\alpha  A x_{k+1}  +(1-\alpha)z_k
        - z_{k+1}). \end{align}
\end{subequations}

Here $\alpha\in (0,2)$ is introduced as  a relaxation parameter \cite{eckstein1992douglas,admm3boyd}.
When $\alpha=1$, the relaxation scheme becomes the standard  ADMM.  The over-relaxation case is 
that $\alpha>1$  and it can accelerate  the convergence toward to the optimal solution \cite{pmlr-v97-yuan19c}.




\subsection{Variants of ADMM and Stochastic ADMM}

Many variants of the classical ADMM have been recently developed.
  These are  two types of common    modifications 
    in many variants of ADMM in order to cater for requirements of different applications. 
  \begin{enumerate}
 \item 
In the linearized  ADMM\cite{goldfarb2013fast}, the augmented Lagrangian function 
 is approximated by  the linearization of quadratic term of $x$ in \eqref{sADMM_x}
and the addition of a proximal term $\frac{\tau}{2}\norm{x-x_k}^2_2$: 
\begin{equation}
\begin{split}
&x_{k+1}: =\argmin_{x} 
\bigg\{
f(x,\xi_{k+1})  + 
\\&  \frac{\tau}{2} 
\norm{x- \left(x_k - \frac{\rho}{\tau} A^\top (Ax_k -y_k+u_k) \right) }^2_2      
\bigg\}.
\end{split}
\end{equation}
 
 \item
The gradient-based ADMM is to solve \eqref{sADMM_x}
inexactly by applying only one step gradient descent for all  $x$-nonlinear terms in $\mathcal{L}_\rho$ 
 with the  step size $1/\tau$: 
$$x_{k+1}   :=x_k  - \frac{1}{\tau} 
\left (f'(x_k,\xi_{k+1}) + \rho A^\top (Ax_k - z_k +u_k) \right).
$$
\end{enumerate}

To accommodate these variants all into one  stochastic setting, 
we formulate a very general scheme to unify   all above cases   
  in the form of stochastic version of ADMM:

 {\bf General stochastic ADMM (G-sADMM)}
 \begin{subequations}
\label{GSADMM}
\begin{align}
x_{k+1} &: =\argmin_{x} 
\hat{\mathcal{L}}_{k+1}(x,z_k,u_k),\\
\label{GsADMM_z}
\begin{split}
z_{k + 1}
&=
\argmin_{y}\bigg\{ g(z) +
\\
&~\quad
\frac{\rho}{2} \left\|   \alpha Ax_{k+1} + (1-\alpha) z_k
   - z + u_k
        \right\|_2^2\bigg\},
        \end{split}
        \\
  \label{GsADMM_u}
u_{k + 1}&=u_k +  (  \alpha Ax_{k+1} + (1-\alpha) z_k  
        - z_{k+1}).
\end{align}
\end{subequations}
 where   the approximate objective function for $x$-subproblem is 
 \begin{equation}
 \begin{split}
\hat{\mathcal{L}}_{k+1}=
&(1-\omega_1)  f(x,\xi_{k+1}) + \omega_1    f'(x_k,\xi_{k+1}) (x-x_k)
\\
&+ (1-\omega)   \frac{\rho}{2}  \| A x - z_k + u_k\|^2_2 
\\
&+\omega  \big( \rho A^\top (A x_k -z_k+u_k) (x-x_k) \big)
\\
&+ \frac{\tau}{2}\norm{x-x_k}_2^2.
\end{split}
\end{equation}
 The explicitness parameters 
$\omega_{1}, \omega \in [0,1]$
and the proximal parameter $\tau \geq 0$.
This scheme \eqref{GSADMM} is very general and 
includes  existing variants  as follows.
\begin{enumerate}
\item  $f(x,\xi)\equiv f(x)$: deterministic version of ADMM: 
\item 
  $\omega_1=\omega=\tau=0$: the  standard stochastic ADMM ({\bf sADMM}); 
 \item  $\omega_1=0$ and $\omega=1$: this scheme is the
stochastic version of the linearized ADMM;
\item $\omega_1=1$ and   $\omega=1$:  this scheme is the
stochastic version of the gradient-based ADMM.
\item  
$\alpha=1$, $\omega_1=1$, $\omega=0$ and $\tau=\tau_k\propto \sqrt{k}$:
the stochastic ADMM considered in \cite{pmlr-v28-ouyang13}.
\end{enumerate}

 \subsection{Main Results}

Define  $V(x) = f(x) + g(Ax)$.  
Let $\alpha\in (0,2)$, $\omega_1, \omega\in \{0,1\}$ and $c=\tau/\rho\geq 0$.
Let $\eps=\rho^{-1}\in (0,1)$.
$\{x_k\}$ denote the sequence of stochastic ADMM \eqref{GSADMM}
with the initial choice $z_0=Ax_0$.
Define $X_t$ as a stochastic process satisfying the SDE
\begin{equation*}
\widehat{M}  dX_t = - \nabla V(X_t) dt + \sqrt{\eps} \sigma(X_t) dW_t
\end{equation*}where the matrix 
 $$
\widehat{M} := 
c+\left(\frac{1}{\alpha}-\omega\right)A^\top A.
$$ and $\sigma$ satisfies 
$$\sigma(x)\sigma(x)^\top = \e_\xi
\left[
\left ( f'(x,\xi)-f'(x)  \right)\left ( f'(x,\xi)-f'(x)  \right)^\top 
\right].$$ Then we have $x_k\to X_{k\eps}$
with a   weak convergence  of order one.
 \subsection{Review and Related Work}

{\it Stochastic and online ADMM}

 The use of stochastic and online techniques 
 for ADMM have recently drawn a lot of interest.
 \cite{WangBanerjeeIMCL2012} first proposed the online ADMM in the standard form, which learns from only one sample (or a small mini-batch) at a time. 
   \cite{pmlr-v28-ouyang13,suzuki2013dual}  proposed the variants of stochastic ADMM
  to attack the difficult nonlinear optimization problem inherent in $f(x,\xi)$   by linearization.
  Very recent, further accelerated algorithms for  the stochastic ADMM have been developed in \cite{zhong2014fast,pmlr-v97-huang19a}

{\it  Continuous models for optimization algorithms}

 In our work,   we focus on the limit of  the stochastic sequence $\{x_k\}$ defined by \eqref{sADMM} 
 and \eqref{GSADMM}   as  $\rho \to \infty$. 
  Define  $$\eps=\rho^{-1}.$$
Assume the proximal  parameter  $\tau$  is  linked to $\rho$ by 
  $\tau= c\rho$ with a constant $c>0$.
 Our interest here is   not about the numerical convergence of $x_k$ from the ADMM towards the optimal point 
 $x_*$ of the objective function as $k\to\infty$ for a {\it fixed } $\rho$, but the proposal of an appropriate continuous model whose 
 (continuous-time) solution $X_t$ is a good approximation to the sequence $x_k$ as $\rho\to\infty$.

 The work in \cite{su2016differential}  is one seminal work based on this perspective of using continuous-time dynamical system tools to analyze various existing discrete algorithms for   optimzation problems to   mode {Nesterov}'s accelerated gradient method.   For the applications to the ADMM,  the recent works in \cite{francca2018admm}
 establishes the first deterministic continuous-time models in the form of ordinary differential equation (ODE)
  for the smooth  ADMM and \cite{pmlr-v97-yuan19c} extends to  the non-smooth case via the differential inclusion model.

In this setting of continuous limit theory,  a  time duration  $T>0$ is fixed first 
so that the continuous-time model is mainly considered in this time interval $[0,T]$.
Usually a small parameter (such as step size) $\eps$ is identified with a correct scaling 
from the discrete algorithm, and used to partition the interval into $K=T/\eps$ windows.
The iteration index $k$  in the discrete algorithm is labelled from $0$ to $K$.  
The convergence of the discrete scheme to the continuous model means that,
with the same initial $X_0=x_0$, for any   $T>0$, as $\eps \to 0 $, then the error between $x_{k}$ and $X_{k\eps}$ measured in certain sense
converges to zero  for any $1\leq k\leq K$.

This continuous viewpoint and formulation has been successful for both deterministic and stochastic optimzation
algorithms in machine learning \cite{e2019continousML}.  The works in
\cite{LTE2017SME,LTE2019JMLR}  rigorously 
present  the mathematical connection of Ito stochastic differential equation (SDE)
with  stochastic gradient descent (SGD) with a step size $\eta$.  More precisely, 
for any small but finite $\eta>0$, the corresponding stochastic differential equation 
carries a small parameter $\sqrt{\eta}$  in its   diffusion terms and  is called 
{\it stochastic modified equation} (SME)  due to the historical reason in numerical analysis
for differential equations. The convergence between $x_k$ and $X_t$ is then formulated in the 
weak sense.
 This SME technique, originally arising from   the numerical analysis of SDE \cite{KPSDEbook2011}, is the major mathematical tool
for most stochastic or online algorithms.

%
%
%
%
%

 \subsection{Contributions}
 
%

%
%

\begin{itemize}
    
    \item 
    We demonstrate how to use mathematical tools like stochastic modified equation(SME) and asymptotic expansion to study the dynamics of stochastic ADMM in the   small 
  step-size (step-size for ADMM is $\eps = 1/\rho$) regime.
    \item We present an unified framework for variants of stochastic version of ADMM, linearized ADMM, gradient-based ADMM, and present a unified stochastic differential equation as their continuous-time limit under weak convergence.
    \item
    We are first to show that the drift term of the stochastic differential equation is the same as the previous ordinary differential equation models.
    \item
    We are first to show that the standard deviation of the solution paths  has the scaling $\sqrt{\eps}$. 
    Moreover,  we can even accurately compute the continuous
limit of the time evolution of  $\eps^{-1/2}\std(x_k)$, 
    $\eps^{-1/2}\std(z_k)$ and $\eps^{-1/2}\std(r_k)$ for the residual $r_k=Ax_k-z_k$. The joint fluctuations of $x, z, r$ is a new phenomenon that has not been studied in previous works on continuous-time analysis of stochastic gradient descent type algorithms.
    \item
    From our stochastic differential equation analysis, we could derive useful insights for practical improvements that are not clear without the continuous-time model. For example, we are able to precisely compute the diffusion-fluctuation trade-off, which would enable us to decide when to decrease step-size and increase batch size to 
accelerate  convergence of stochastic ADMM.
\end{itemize}

\subsection{Notations and Assumptions}
We use $\norm{\cdot}$ to denote the Euclidean two norm  if the subscript is not specified.
and all vectors are referred as column vectors. 
$f'(x,\xi)$, $g'(z)$ and $f''(x,\xi)$, $g''(z)$  refer  to the first (gradient) and second (Hessian) derivatives w.r.t. $x$.

The first  assumptions is 
{\bf Assumption I}:
   $f(x)$, $g$  and for each $\xi$, , $f(x,\xi)$, are closed proper convex functions;
   $A$ has full column rank.

Let $\mathcal{F}$ as  the set of functions of at most polynomial growth,
  $\varphi \in  \mathcal{F}$ if there exists constants  $C_1$, $\kappa$ > 0 such that
  \begin{equation}
  \label{fun:poly}
  |\varphi(x)|< C_1 (1+\norm{x}^\kappa)
  \end{equation}
 
 To apply the SME theory, we need the following assumptions \cite{LTE2017SME,LTE2019JMLR}
  {\bf Assumptions II}: 
    \begin{enumerate}[(i)]
  \item  $f(x)$, $f(x,\xi)$ and $g(z)$ are differentiable 
  and the second order derivative $f'', g'' $  are uniformly bounded in $x$, and 
 almost surely in $\xi$  for $f(x,\xi)$.
   $\e \norm{f'(x,\xi)}_2^2  $ is uniformly bounded in $x$.
 
  \item
  $f(x)$, $f(x,\xi)$, $g(x)$ and the partial derivatives up to order $5$ belong to $\mathcal{F}$
  and for $f(x,\xi)$, it means the almost surely in $\xi$, i.e. ,
  the constants $C_1$, $\kappa$ in \eqref{fun:poly} do not depend on $\xi$.
  \item  $f'(x)$ and $f'(x,\xi)$   
  satisfy a uniform growth condition: $\norm{f'(x)} + \norm {f'(x,\xi)}\leq C_2 (1+\norm{x})$ for a constant
  $C_2$ independent of $\xi$. 
  \end{enumerate}
  The conditions  (ii) and (iii) are inherited from \cite{LTE2017SME,Milstein1986} , which might 
  be relaxed in certain cases. Refer to remarks  in Appendix C of  \cite{LTE2017SME}.

\section{Weak Approximation to Stochastic ADMM}

In this section, we show the weak approximation to the stochastic ADMM \eqref{sADMM}
and the general family of stochastic ADMM variant \eqref{GSADMM}.
Appendix \ref{app:A} is a summary of the background of the weak approximation and the stochastic modified equation
for interested readers.

Given the noisy gradient  $f'(x,\xi)$ and its expectation $f'(x)=\e f(x,\xi)$,
we define the following matrix $\sigma(x)\in \RR^{d\times d}$ by 
 \begin{equation}\label{Sigma}
 \begin{split}
 \Sigma(x)&=\sigma(x)\sigma(x)^\top 
 \\
 &= \e_\xi
\left[
\left ( f'(x,\xi)-f'(x)  \right)\left ( f'(x,\xi)-f'(x)  \right)^\top 
\right].
\end{split}
\end{equation}
 
\begin{theorem}[SME for sADMM]
\label{Sthm}
Consider the standard stochastic ADMM 
without relaxation \eqref{sADMM} with $\alpha=1$.
Let $\eps=\rho^{-1}\in (0,1)$.
$\{x_k\}$ denote the sequence of stochastic ADMM 
with the initial choice $z_0=Ax_0$.
 
Define $X_t$ as a stochastic process satisfying the SDE
\begin{equation}\label{SDE}
(A^\top A)\, dX_t = - \nabla V(X_t) dt + \sqrt{\eps} \sigma(X_t) dW_t
\end{equation}
where $V(x)=\e_\xi V(x,\xi)=\e_\xi f(x,\xi)+g(Ax)$ and the diffusion matrix $\sigma$ is  defined by \eqref{Sigma},
Then we have $x_k\to X_{k\eps}$
with the  weak convergence  of order $1$.
\end{theorem}

\begin{proof}[Sketch of proof]
The ADMM scheme is in a form of the iteration  of the triplet $(x,z,\lambda)$ where $\lambda=\eps u$.
But by the first order optimality condition for $z$-subproblem and $u$-subproblem, 
we have $\lambda_{k+1} = g'(z_{k+1})$   for whatever input triplet $(x_k,z_k,\lambda_k)$. 
Thus, the variable $\lambda$ is faithfully  replaced by   $g'(z)$. The remaining goal is to further 
replace the $z$ variable by the $x$ variable  that the ADMM iteration is approximately  reduced to the iteration
only for $x$ variable.   This is indeed true  because of the 
 critical observation (Proposition \ref{p:rp2})   that   the residual $r_k=Ax_k-z_k$ is  has a second order smallness, 
 belonging to  $\mathcal{O}(\eps^2)$, if $r_0=Ax_0-z_0=0$. 
Thus,   ADMM is transformed into the one-step iteration  form \eqref{eq:AAite} only in $x$ variable
with  $ \mathcal{A}(\eps, x,\xi) =  f'(x,\xi) + A^\top g'(Ax) + \mathcal{O}(\eps)$.
The conclusion then follows by directly checking the conditions  \eqref{con:mm} in Theorem \ref{thm:Milstein}.
\end{proof}

Our   main theorem is for the G-sADMM scheme which 
contains the relaxation parameter $\alpha$, the proximal parameter $c$ and the implicitness parameters $\omega,\omega_1$.

\begin{theorem}[SME for G-sADMM]
\label{Gthm}
Let $\alpha\in (0,2)$, $\omega_1, \omega\in \{0,1\}$ and $c=\tau/\rho\geq 0$.
Let $\eps=\rho^{-1}\in (0,1)$.
$\{x_k\}$ denote the sequence of stochastic ADMM \eqref{GSADMM}
with the initial choice $z_0=Ax_0$.

Define $X_t$ as a stochastic process satisfying the SDE
\begin{equation}\label{SDE}
\widehat{M}  dX_t = - \nabla V(X_t) dt + \sqrt{\eps} \sigma(X_t) dW_t
\end{equation}where the matrix 
\begin{equation}
\label{def:Mhat}
\widehat{M} := 
c+\left(\frac{1}{\alpha}-\omega\right)A^\top A.
\end{equation}

Then we have $x_k\to X_{k\eps}$
in  weak convergence  of order $1$, with
the following precise meaning.

For any time interval $T>0$ and for any test function $\varphi$ 
such that $\phi$ and its partial derivatives up to order $4$ belong to 
$\mathcal{F}$,
there exists a constant $C$ such that 
\begin{equation}
\left | \e \varphi(X_{k\eps})-\e \varphi(x_{k})
\right| \leq C \eps , ~~ k \leq   \lfloor
T/\eps\rfloor
\end{equation}
\end{theorem}
\begin{proof}[Sketch of proof]
The idea of this proof is similar to that in Theorem \ref{Sthm}  
even with the introduction of $c,\omega,\omega_1$ parameters.
But for the relaxation parameter  when $\alpha \neq 1$, 
we need to overcome a substantial challenge.
If $\alpha\neq 1$, then the   residual $r_k=Ax_k-z_k$ is now only  at order  $\mathcal{O}(\eps)$,
not $\mathcal{O}(\eps^2)$. In the proof, we propose a new $\alpha$-residual  
$ \widehat{r}^\alpha_{k+1}:=\alpha r_k+(\alpha-1) (z_{k+1}-z_k)$ and show that
it is indeed as small as $\mathcal{O}(\eps^2)$ (Proposition \ref{p:ar}) to solve this challenge.
The difference between $r_k$ and the $\alpha$-residual thus induces the extra $\alpha$-term 
in the new coefficient matrix $\widehat{M}$ in \eqref{def:Mhat}.
\end{proof}
The rigorous proof is in Appendix \ref{sec:pfthm}.

\begin{remark}
We do not present a simple form of SME as the the second order weak approximation 
as for the SGD scheme, due to the complicated issue of the residuals.
In addition, the proof requires a regularity  condition for the functions $f$ and $g$;
at least $g$ needs to have the third order derivatives of $g$. So, 
our theoretic theorems can not cover the non-smooth function $g$.
Our numerical tests suggest that the conclusion holds too for $\ell_1$ regularization function $g(z)=\norm{z}_1$.
\end{remark}

\begin{remark}
In general applications, it is very difficulty to get the expression of the variance matrix
$\Sigma(x)$ as a function of $x$, except in very few simplified cases. 
In applications of empirical risk minization,
the function $f$ is the empirical average of the loss on each sample $f_i$:
$f(x)=\frac{1}{N}\sum_{i=1}^N f_i(x)$.
The diffusion matrix $\Sigma(x)$ in \eqref{Sigma} becomes the following form
\begin{equation}
\Sigma_N(x)= \frac{1}{N}\sum_{i=1}^N (f'(x)-f'_i(x))f'(x)-f'_i(x)^\top.
\end{equation}
It is clear that if $f_i(x)= f(x,\xi_i)$ with $N$ iid samples $\xi_i$,
then $\Sigma_N(x)\to \Sigma(x) $ as  $N\to \infty$.
\end{remark}

\begin{remark}
The stochastic scheme \eqref{GSADMM} 
is the simplest form of using only one instance of the
 gradient    $f'(x,\xi_{k+1})$ in each iteration.
 If a batch size larger than one is used, then
 the one instance gradient 
 $f'(x,\xi_{k+1})$ is replaced by the average 
 $\frac{1}{B_{k+1}}\sum_{i=1}^{B_{k+1}} f'(x,\xi_{k+1}^i)$
 where $B_{k+1}>1$ is the batch size and $(\xi_{k+1}^i)$
 are $B_{k+1}$ iid samples.
 Under these settings,
 $\Sigma$ should be multiplied by a fact
 $\frac{1}{B_t}$ where the continuous-time  function $B_t$ 
 is the linear interpolation of $ B_{k}$ at times $t_k=k\eps$.
 The stochastic modified equation \eqref{SDE} is then in the following form
 $
 \widehat{M} dX_t = - \nabla V(X_t) dt +  \sqrt{\frac{\eps}{B_t}}\sigma(X_t) dW_t.
$
 \end{remark}

Based on the SME above, we can find the 
  stochastic asymptotic expansion of $X^\eps_t$ 
\begin{equation}
\label{eq:Xexp12}
X^\eps_t \approx  X^0_t + \sqrt{\eps} X^{(1/2)}_t+ \eps X^{(1)}_t+\ldots.
\end{equation}
 See  Chapter 2 in \cite{FW2012} for rigorous justification.
$X^0_t$ is {\it deterministic} as  the gradient flow of the deterministic problem:
$\dot{X}^0_t=- V'(X^0_t)$,  $X^{(1/2)}_t$ and $X^{(1)}_t$ are {\it stochastic}
and satisfy certain SDEs independent of $\eps$.
The useful conclusion   is   that 
 the standard deviation of $X^\eps_t$,  mainly coming  from the term $\sqrt{\eps}X^{(1)}_t$, is $\mathcal{O}(\sqrt{\eps})$.
Hence, the standard deviation of  the stochastic ADMM $x_k$ is   $\mathcal{O}(\sqrt{\eps})$
and more importantly,  the rescaled two standard deviations  $\eps^{-1/2}\std(x_k)$ and $\eps^{-1/2}\std(X_{k\eps})$ 
are close  as the function of the   time $t_k=k\eps$.

We can investigate the fluctuation of the $z_k$ sequence  generated by the stochastic ADMM.
The approach is to study  the modified equation of its continuous version $Z_t$ first.
Since the residual $r=Ax-z$ is on the order $\mathcal{O}(\eps)$ shown in the appendix (Proposition \ref{p:rpp}
and \ref{p:rp2}), we have the following result.

\begin{theorem}
\label{Zthm}
~
\begin{enumerate}[(i)]
\item There exists a deterministic function $h(x,z)$ such that  
\begin{equation}
\dot{Z}^\eps_t=A\dot{X}^\eps_t+\eps h(X^\eps_t, Z^\eps_t)
\end{equation}
where $X^\eps_t$ is the solution to the SME in Theorem \eqref{Gthm}
and $\{z_k\}$ is a weak approximation to $\{Z^\eps_t\}$ with the order 1.
\item In addition, we have the following asymptotic for $Z^\eps_t$:
\begin{equation}
Z^\eps_t
 \approx
AX^0_t + \sqrt{\eps} AX^{(1/2)}_t + \eps Z^{(1)}_t 
\end{equation}
where ${Z}^{(1)}_t$ satisfies $\dot{Z}^{(1)}_t= h(X^0_t,AX^0_t)$.
\item The standard deviation of   $z_k$  is  on the  order 
$\sqrt{\eps}$.
\end{enumerate}
\end{theorem}

Recall the residual $r_k=Ax_k-y_k$ 
and in view of Corollary \ref{cor:rk} in the appendix,
we have the following result that there exists a function $h_1$ such that 
\begin{equation}
\begin{split}
\alpha R^\eps_t =   (1-\alpha)  ({Z}^\eps_t - Z^\eps_{t-\eps})+\eps^2 h_1(X^\eps_t,Z^\eps_t)
\end{split}
\end{equation}
and the residual $\{r_k\}$ is a weak approximation to $\{R^\eps_t\}$ with the order 1.
If $\alpha= 1$ in the G-sADMM \eqref{GSADMM}, then the expectation and standard deviation of $R_t$ and $r_k$  are both at order $\mathcal{O}(\eps^2)$.
If $\alpha\neq 1$ in the G-sADMM \eqref{GSADMM}, then the expectation and standard deviation of  $R_t$ and $r_k$  are only at order $\mathcal{O}(\eps)$.

%
%

\section{Numerical Examples}

{\bf Example 1: one dimensional   example}
In this simple example,  the dimension $d=1$. 
Consider   $f(x,\xi)= (\xi+1)x^4+ (2+\xi) x^2-(1+\xi) x $,   
where  $\xi$ is a  Bernoulli random variable  taking  values  $-1$ or $+1$ 
with equal probability.
We test $g(z)=z^2 $ and $g(z)=\abs{z}$. 
The matrix $A=I$.
These settings satisfy the assumptions in our main theorem.
  We choose $c =\omega$ such that   $\widehat{M}=\frac{1}{\alpha}$.
The SME  when  $g(z)=z^2$
 is 
$\frac{1}{\alpha}dX_t =  -  (4x^3+6x-1)dt  + \sqrt{\eps} \abs{4x^3+2x-1}dW_t $.
The choice of the initial guess is 
$x_0=z_0=1.0$ and $\lambda_0=g'(z_0)$.
The terminal time $T= 0.5$ is fixed.

Figure \ref{fig:toy-mean-std-g2} shows the match of the 
expectation and the standard deviation of the sequence 
$x_k$ of stochastic ADMM and  $X_{t_k}$
of the SME with $t_k=k\eps$.
Furthermore, we plot Figure     $400$ random trajectories from both models
in Figure \ref{fig:toy-admm-sme-traj}. 
and it shows the fluctuation in the sADMM can be well capturedd by the SME model.

The acceleration effect of $\alpha$ for the deterministic ADMM has been shown in \cite{pmlr-v97-yuan19c}.
Figure \ref{fig:toy-meanx-g12} confirms the same effect both for smooth and non-smooth $g$
for the expectation of the solution sequence $x_k$.
\begin{figure}
\begin{center}
\includegraphics[width=0.4\textwidth]{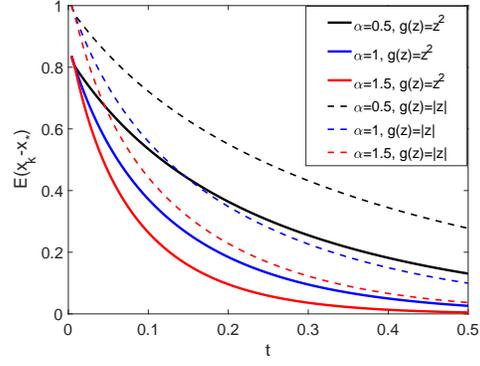}
\end{center}
\caption{The expectation of $x_k-x_*$ w.r.t. $\alpha$. $x_*$ is the true minimizer. 
The result is based on the average of 10000 runs.}
\label{fig:toy-meanx-g12}
\end{figure}

The SME  does not only provide the expectation of the solution, but also
provides the fluctuation of the numerical solution $x_k$ for any given $\eps$.
Figure  \ref{fig:toy-mean-std-g2} compares the mean and standard deviation (``std'') 
between $x_k$ and $X_{k\eps}$ at $\eta=2^{-7}$.
The right vertical axis is the value of standard deviation and the two $\std$ curves 
are very close.   In addition, with the same setting, a few hundreds of 
trajectory samples $x$ are shown together in Figure \ref{fig:toy-admm-sme-traj},
which illustrate the match both in the mean and in the std between
the stochastic ADMM and the SME.

\begin{figure}
\begin{center}
\includegraphics[width=0.4\textwidth]{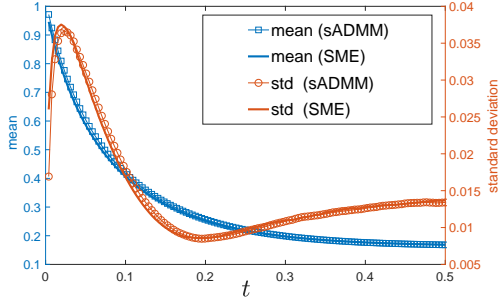}
\caption{ The expectation (left axis) and standard deviation (right axis) of $x_k$ (from stochastic ADMM) and $X_t$
(from stochastic modified equation) .  $\eps=2^{-7}$. The results are based on the 
average of $10000$ independent runs. The over-relaxation parameter $\alpha=1.5$ is used.
 }
 \label{fig:toy-mean-std-g2}
 \end{center}
\end{figure}

\begin{figure}
\begin{center}
\includegraphics[width=0.4\textwidth]{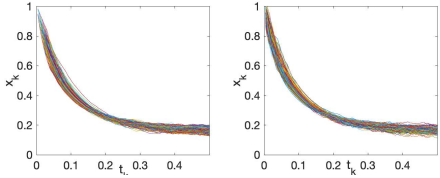}
\caption{ The $400$ sample trajectories from stochastic ADMM (left) and 
SME (right).   
 }
 \label{fig:toy-admm-sme-traj}
 \end{center}
\end{figure}

To verify our theorem on the convergence order,
a test function $\varphi(x)= x+x^2 $ is used for the test
of the weak convergence error: 
$$err: = \max_{1\leq k\leq  \lfloor T/\eps \rfloor}\abs{ \e \varphi(x_{ k})-\e \varphi(X_{k\eps}) }.$$

For each $m=4,5,\ldots, 11$,  
set
 $\rho =2^m/T$, so $\eps = T 2^{-m}$ and $k=1,2\ldots, 2^m$.
Figure \ref{fig:toy-weak-error} shows the error  $  err_m$
versus  $m$  in the semi-log plot for three values of 
relaxation parameter $\alpha$. The first order convergence rate 
$err_m \propto  \eps$ is verified.

\begin{figure}
\begin{center}
\label{fig:toy-weak-error}
\includegraphics[width=0.4\textwidth]{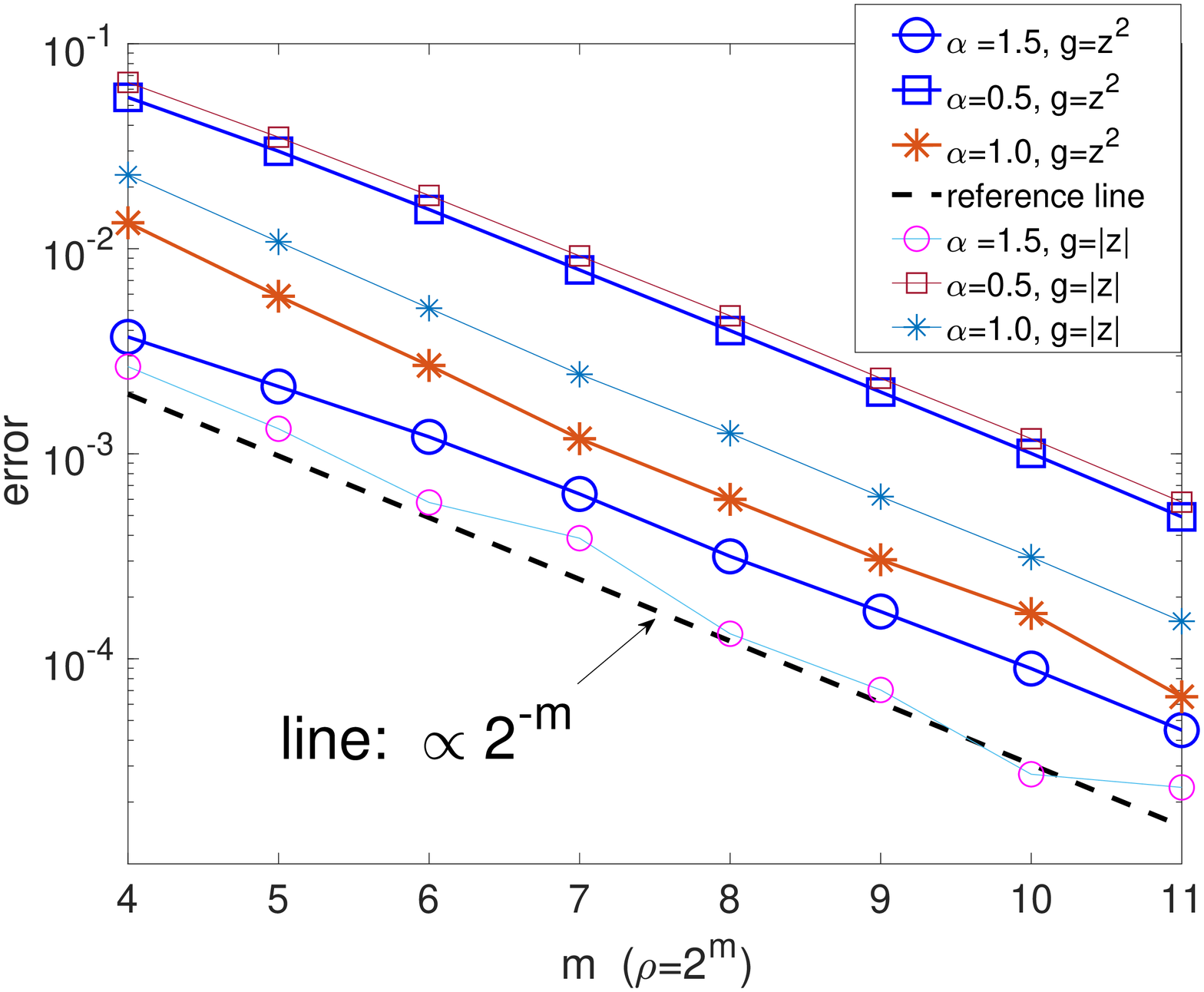}
\caption{(Verification of the first order approximation )The weak convergence error $err_m$ versus $m$ for 
various $\alpha$ and  $\ell_2$, $\ell_1$ regularization $g$.
The step size $\eps=1/\rho=2^{-m}T$.
$T=0.5$. The result is based on the average of $10^5$ independent runs.  }
\end{center}

\end{figure}

We also numerically investigated   the convergence rate 
for  the non-smooth penalty $g(z)=|z|$, even though this $\ell_1$ regularization function does not satisfy our assumptions. 
The diffusion term $\Sigma(x)$ is still the same as in the $\ell_2$  case since $g(z)$ is deterministic. 
For the corresponding SDE, at least formally, we can write 
$\frac{1}{\alpha}dX_t =   -  (4x^3+4x-1+\mbox{sign}(x))dt  + \sqrt{\eps} \abs{4x^3+2x-1}dW_t$,
by using the sign function as  $g'(z)$. The rigorous meaning needs the concept of stochastic differential inclusion, which
is out of the scope of this work.
The numerical results in Figure \ref{fig:toy-weak-error} shows that
the weak convergence order $1$ is also true for this $\ell_1$ case.

Finally, we   test the  orders  for the standard deviation of  $x_k$ and $z_k$. 
Th consistence of $\std(x_k)$ with the SME's $\std(X_{k\eps})$ has been shown in
Figure \ref{fig:toy-mean-std-g2}.
The theoretic  prediction is that both are at order $\sqrt{\eps}$.
We plot the  sequences of $\eps^{-1/2} \std(x_k) $ and 
$\eps^{-1/2} \std(z_k) $ for various $\eps$.
These two quantities  should be the same regardless of $\eta$,
and only depends on $\alpha$.
which is confirmed by Figure \ref{fig:toy-stdxz-g2}.
\begin{figure}
\begin{center}
\includegraphics[width=0.3\textwidth]{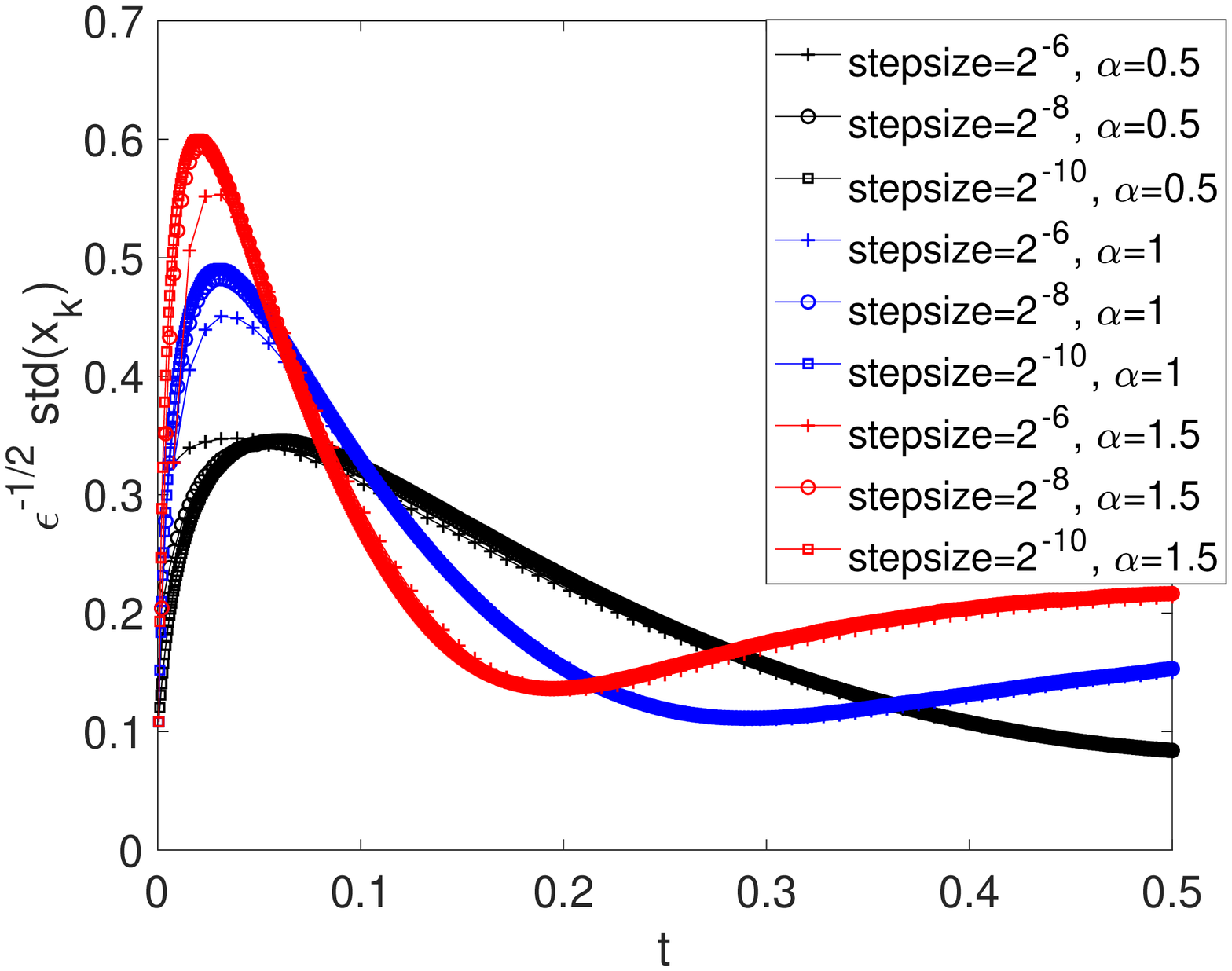}
\includegraphics[width=0.3\textwidth]{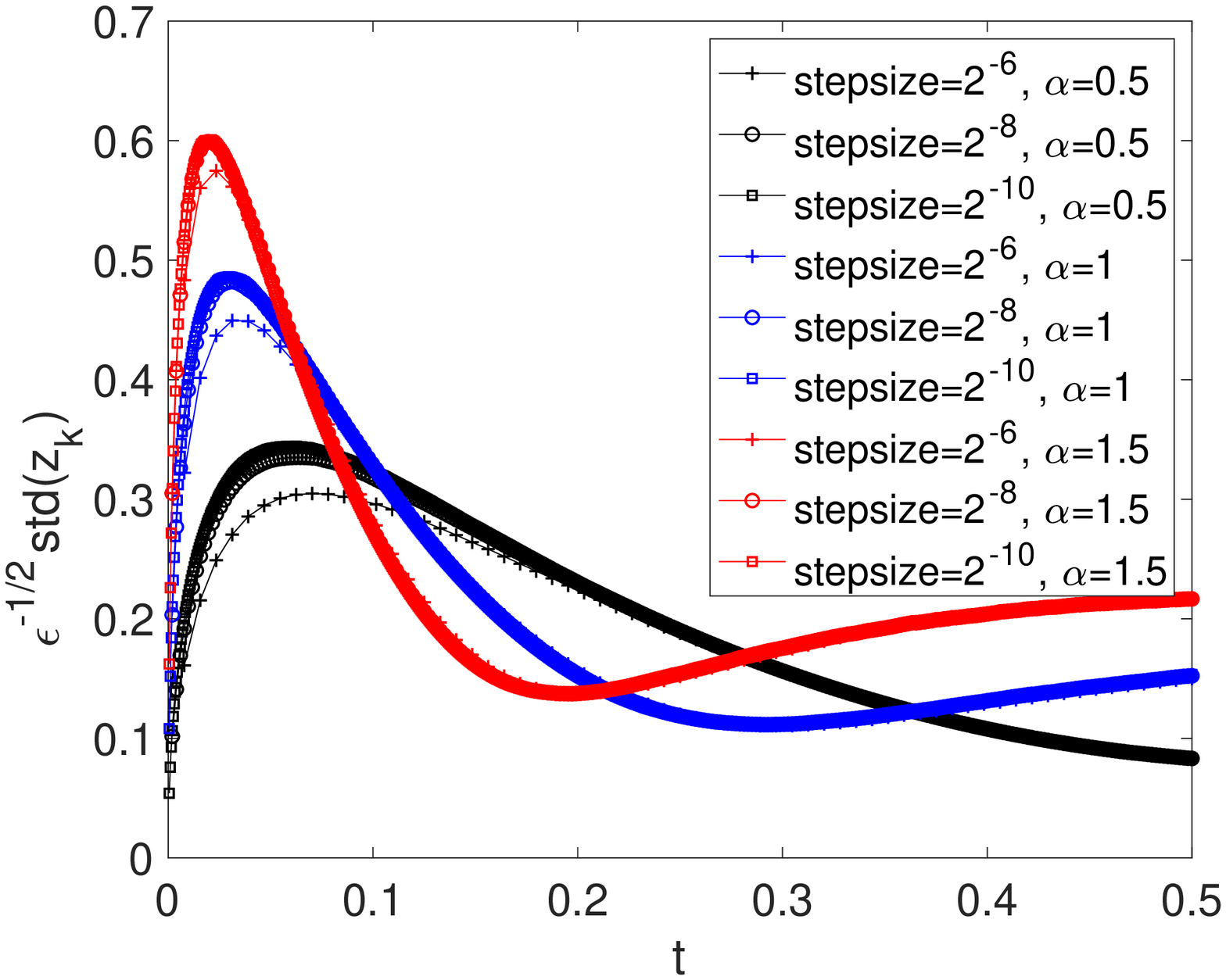}
\end{center}
\caption{std of $x_k$ and $z_k$}
\label{fig:toy-stdxz-g2}
\end{figure}
 For the residual, the theoretic prediction is that both 
 $\e r_k$ and $\std r_k$ are on the order $\eps^{-1}$ if $\alpha\neq 1$.
 We plot  $\eps^{-1} \e (r_k)$, $ \eps^{-1}\std(r_k)$,
against the time $t_k=k\eps$ in Figure \ref{fig:toy-meanrk} and Figure \ref{fig:toy-stdrk}, respectively.
For the stochastic ADMM scheme with $\alpha=1$, 
the numerical test shows that 
$\e r_k$ and $\std r_k$ are on the order $\eps^{-2}$.

\begin{figure}
\begin{center}
\includegraphics[width=0.3\textwidth]{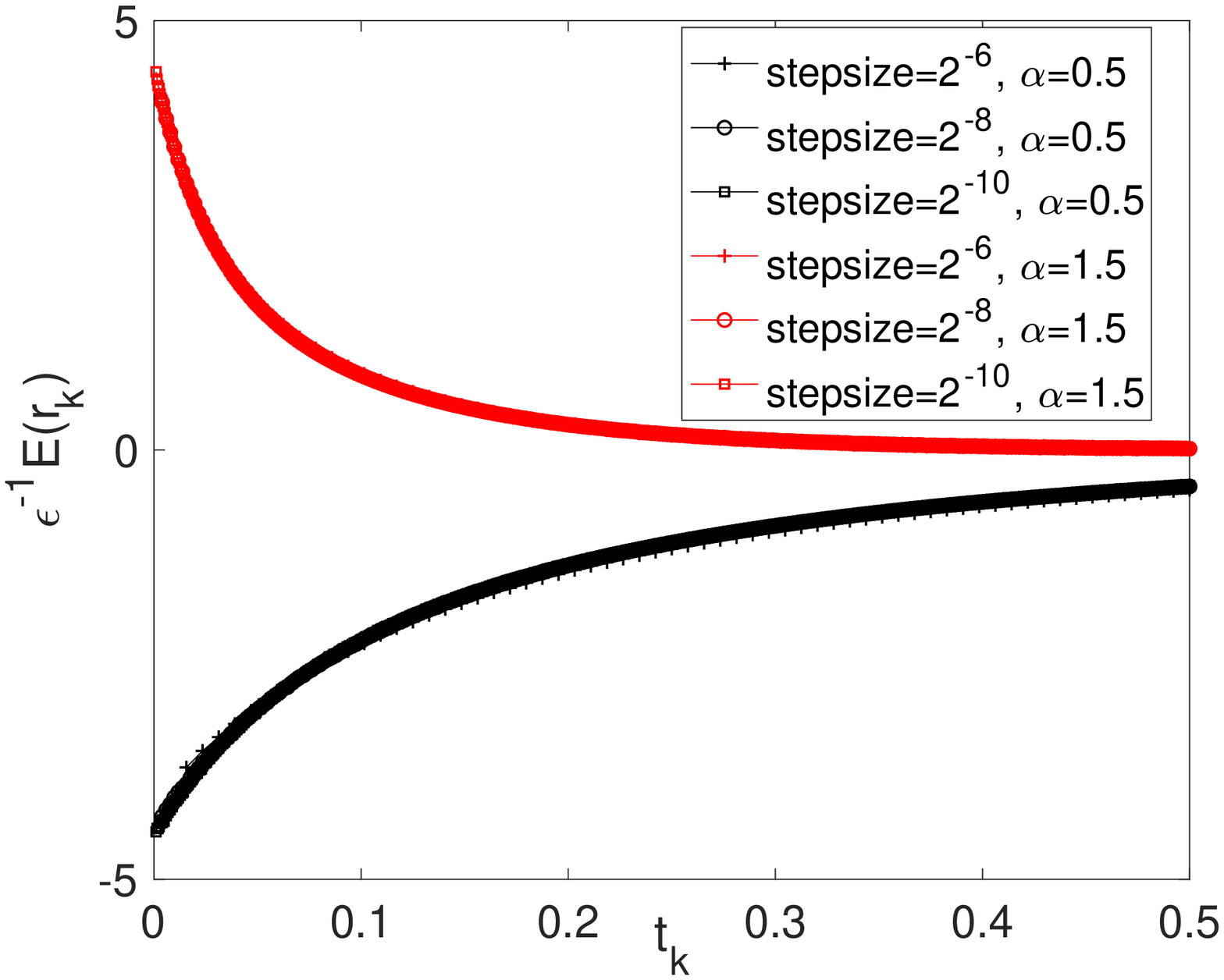}
\includegraphics[width=0.3\textwidth]{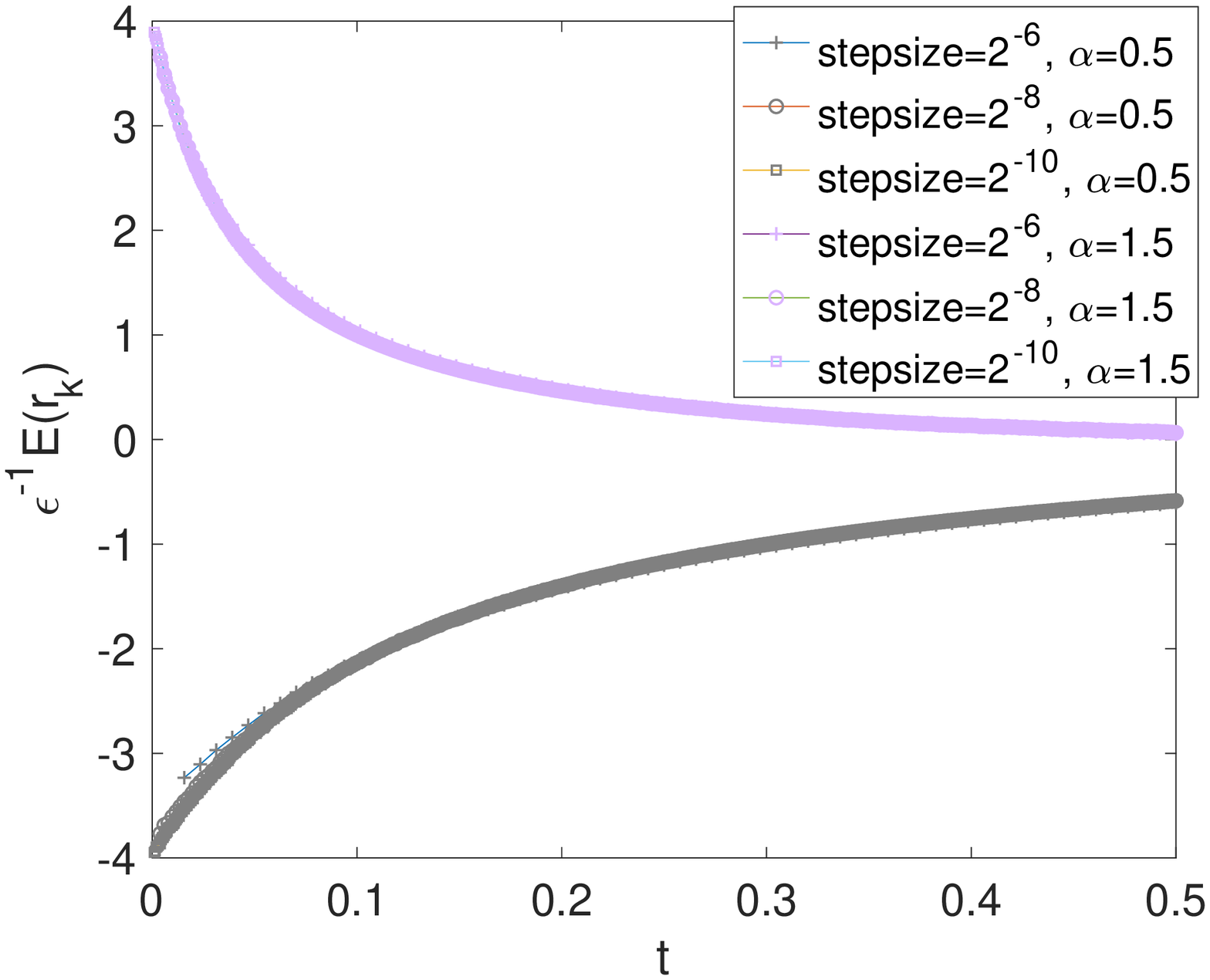}
\end{center}
\caption{The verification of the mean   residual $r_k =\mathcal{O}(\eps^{-1})$ for $\alpha\neq 1$.
$g(z)=z^2$ (top) and $g(z)=|z|$(bottom)}
\label{fig:toy-meanrk}
\end{figure}

\begin{figure}
\begin{center}
\includegraphics[width=0.3\textwidth]{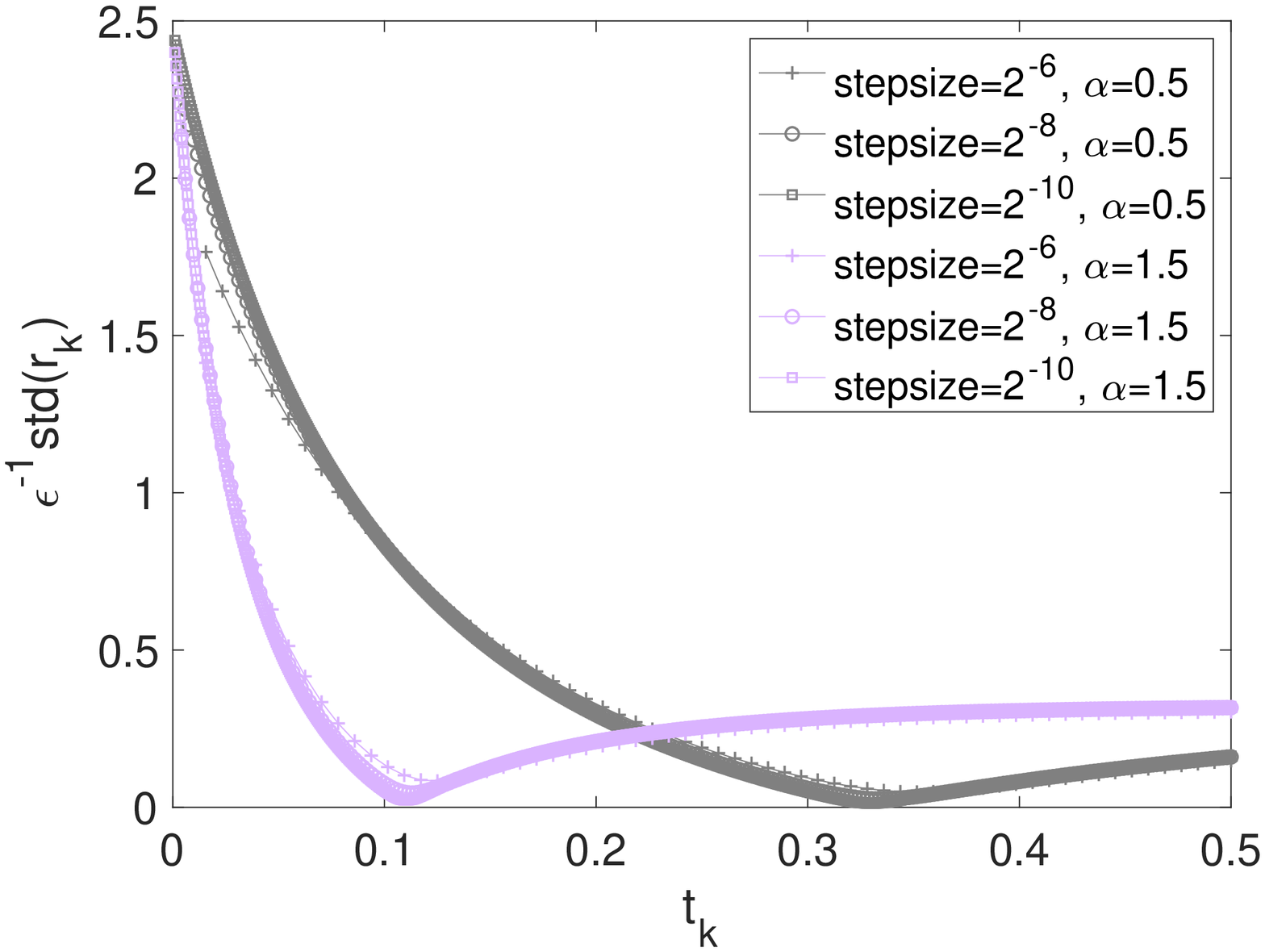}
\includegraphics[width=0.3\textwidth]{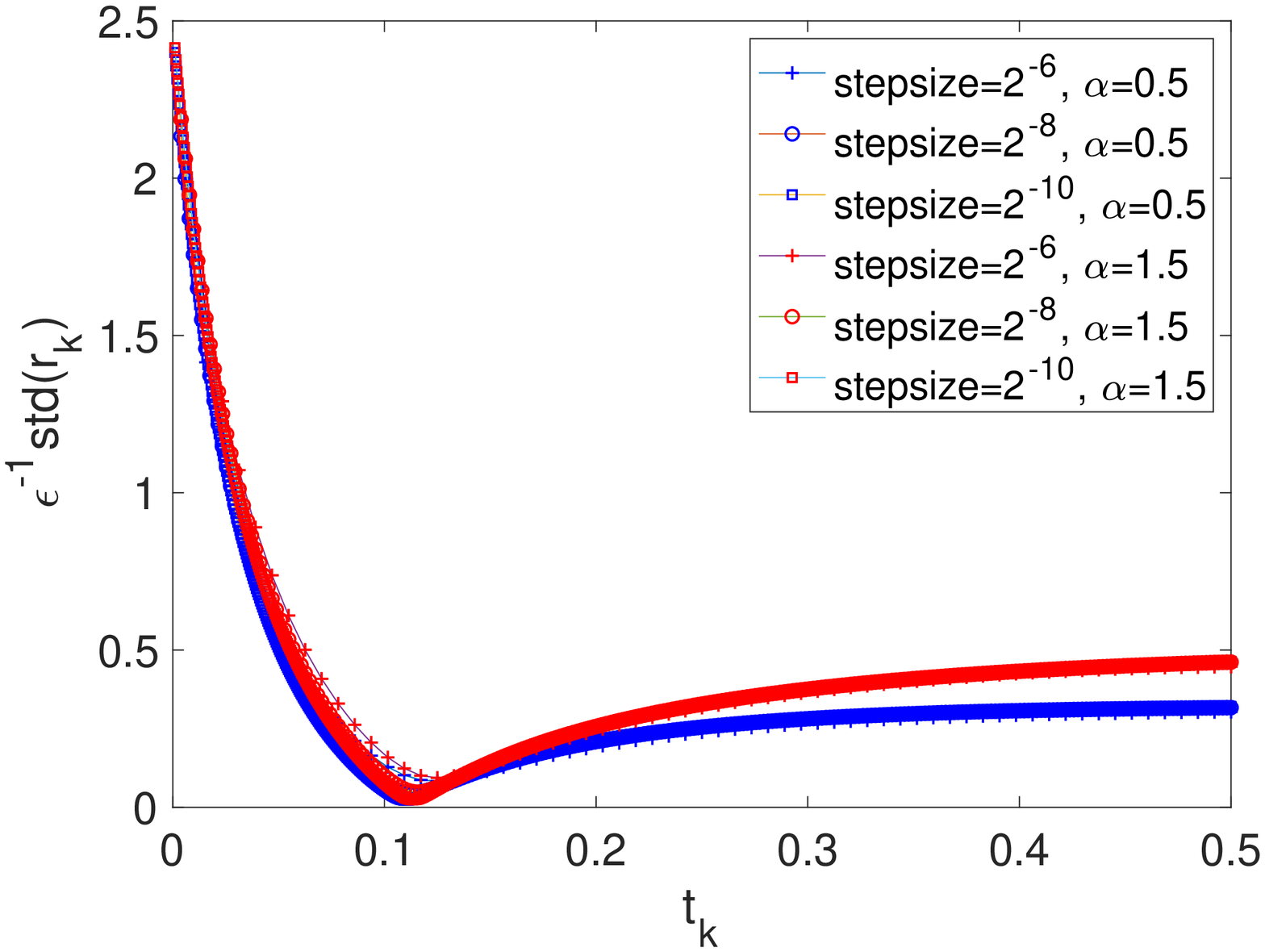}
\end{center}
\caption{The verification of the std of the residual $r_k \sim \eps^{-1}$ for $\alpha\neq 1$.
$g(z)=z^2$ (top) and $g(z)=|z|$(bottom).}
\label{fig:toy-stdrk}
\end{figure}

\begin{figure}
\begin{center}
\includegraphics[width=0.3\textwidth]{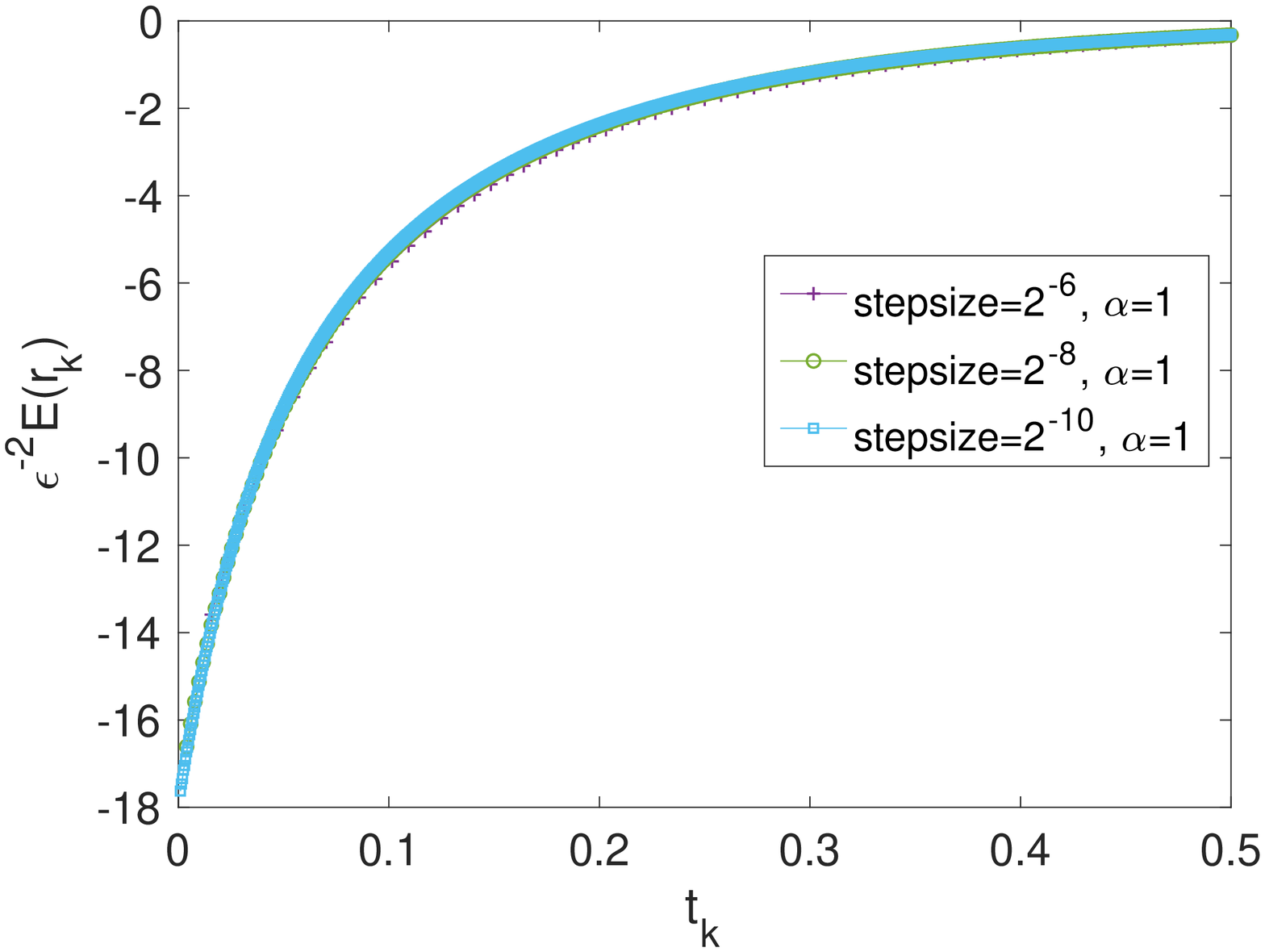}
\includegraphics[width=0.3\textwidth]{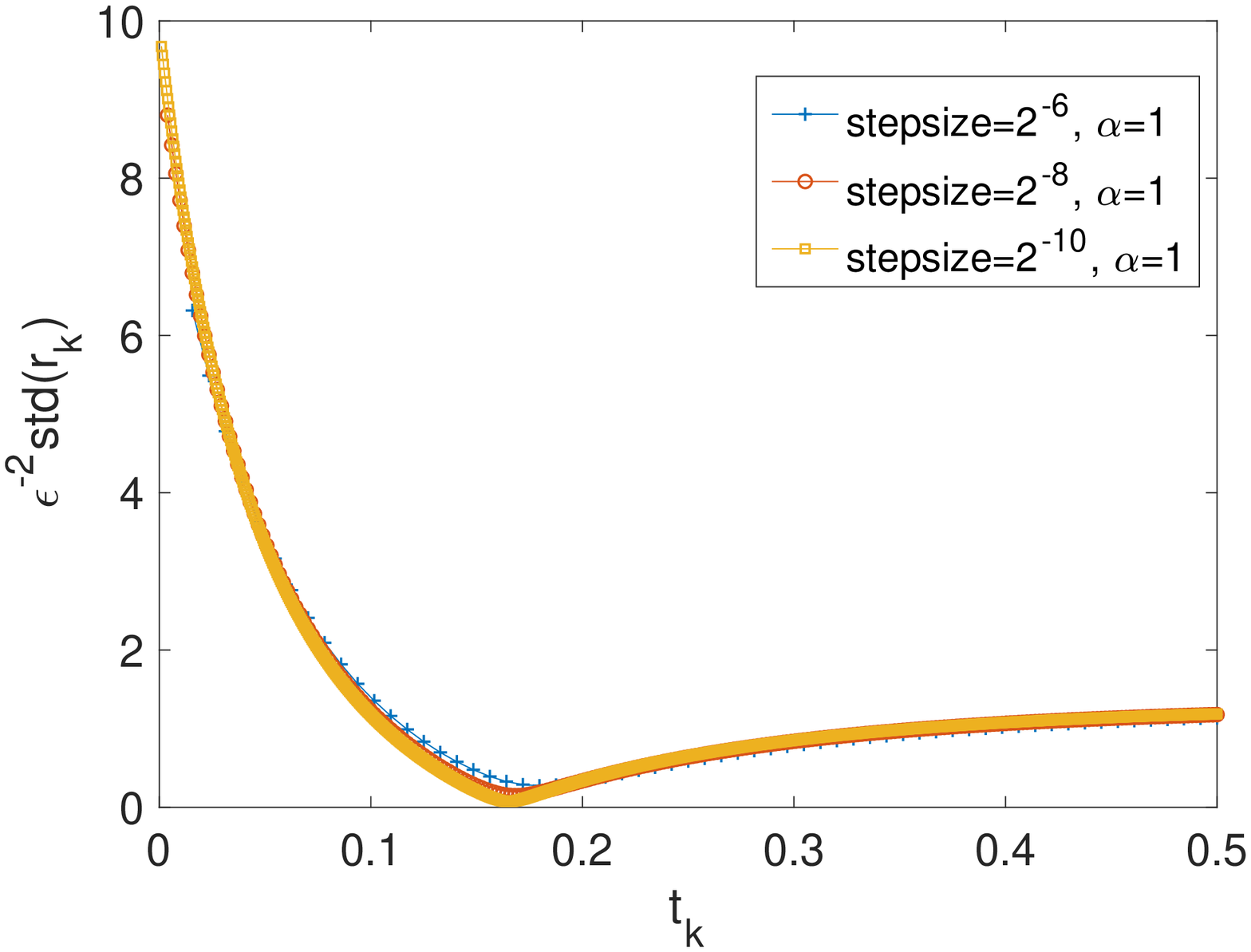}
\end{center}
\caption{The mean (top)  and std (bottom) of the residual $r_k \sim \eps^{-2}$ for  the scheme without relaxation $\alpha=1$.
$g(z)=z^2$.}
\label{fig:toy-stdrk}

\end{figure}

{\bf Example 2: generalized ridge and lasso regression}
We perform experiments on the generalized ridge regression.
\beq \begin{aligned}
&\underset{x\in \RR^d, z\in \RR^m}{\text{minimize}} && 
 \frac12 \e_{\boldsymbol \xi} \left( \xi_{in}^\top x - \xi_{obs} \right)^2 
+ g(z)
\\
&\text{subject to} &&\  Ax-z = 0.
\end{aligned}
\eq
where $g(z)= \frac12 \beta \norm{z}^2_2$ (ridge regression) or $g(z)=\beta \norm{z}_1$
(lasso regression), with  a constant $\beta>0$.
  $A$ is a penalty matrix specifying the desired structured   pattern of $x$.
 Among the random $\boldsymbol \xi=(\xi_{in}, \xi_{obs})\in \RR^{n+1}$,
$\xi_{in}  $ is the
zero-mean random (column) vector with  uniformly distribution  in the hypercube $  (-0.5,0.5)^d $
with independent components. 
The labelled data $\xi_{obs} := \xi_{in}^\top v + \zeta$, where $v\in \RR^n$ is a  given    vector
and $\zeta =  \mathcal{N}(0, \sigma_\zeta^2)$ is the zero-mean measurement noise, independent of $\xi_{in}$.
The analytic expression of the matrix-valued function $\Sigma(x)$ is available
based on the four-order momentums of   $\xi_{in}$.

We use a batch size $B$ for the stochastic ADMM   ($B=9$ is used in experiments).
Then the corresponding SME  for the ridge regression problem
is 
$$\widehat{M}dX_t = - \Omega (X_t -v)dt -\beta A^\top A X_t \,  dt + \sqrt{\eps/B}\Sigma^{1/2}(X_t)dW_t  $$

The SME  for the lasso regression (formally) is 
$$\widehat{M}dX_t \in  - \Omega (x-v) dt-\frac12 \beta A^\top \mbox{sign}( A x)  dt + \sqrt{\eps/B}\Sigma^{1/2}dW_t  $$
The direct simulation of these stochastic equations has a high computational burden
because of the complexity of matrix square root for $\Sigma(x)$.
So, our tests are only restricted to   the dimension $d=3$. 

Set  $A$ is the  {Hilbert matrix} multiplied by $0.5$.
$\sigma_\zeta^2=0.1$. $\beta=0.2$.
The vector $v$ is set as $\mbox{linspace}(1,2,d)$.
 The initial $X_0=x_0$ is the zero vector.
$z_0=Ax_0$.

In algorithms, set $c=1$.
We choose the test function $\varphi(x)=\sum_{i=1}^d x_{(i)}$.
Denote $\varphi_k =\varphi(x_k)$ where $x_k$ are 
the sequence computed from the (unified) stochastic ADMM
with the batch size $B$ .
Denote  $\Phi_{k\eps}=\varphi(X_{k\eps})$ where $X_t$ is the solution of the SME.

  Let $\alpha= 1.5$, $\omega=1$, $\omega_1=1$.  $T=40$.
We first show in Figure \ref{fig:reg-mean}  the mean of $  \phi_k $ and $   \Phi_{k\eps}$
versus the time $t_k=k\eps$, for a fixed  $\eta=2^8$.
To test the match of the fluctuation, we plot  in 
Figure \ref{fig:ridge-std} the sequence 
$\eps^{-1/2}\std(\varphi_k)$ and 
$\eps^{-1/2}\std(\Phi_k)$
for three different values of $\eps = 2^{-m}T$ with $m=6,7,8$.

\begin{figure}
\begin{center}
\includegraphics[width=0.34\textwidth]{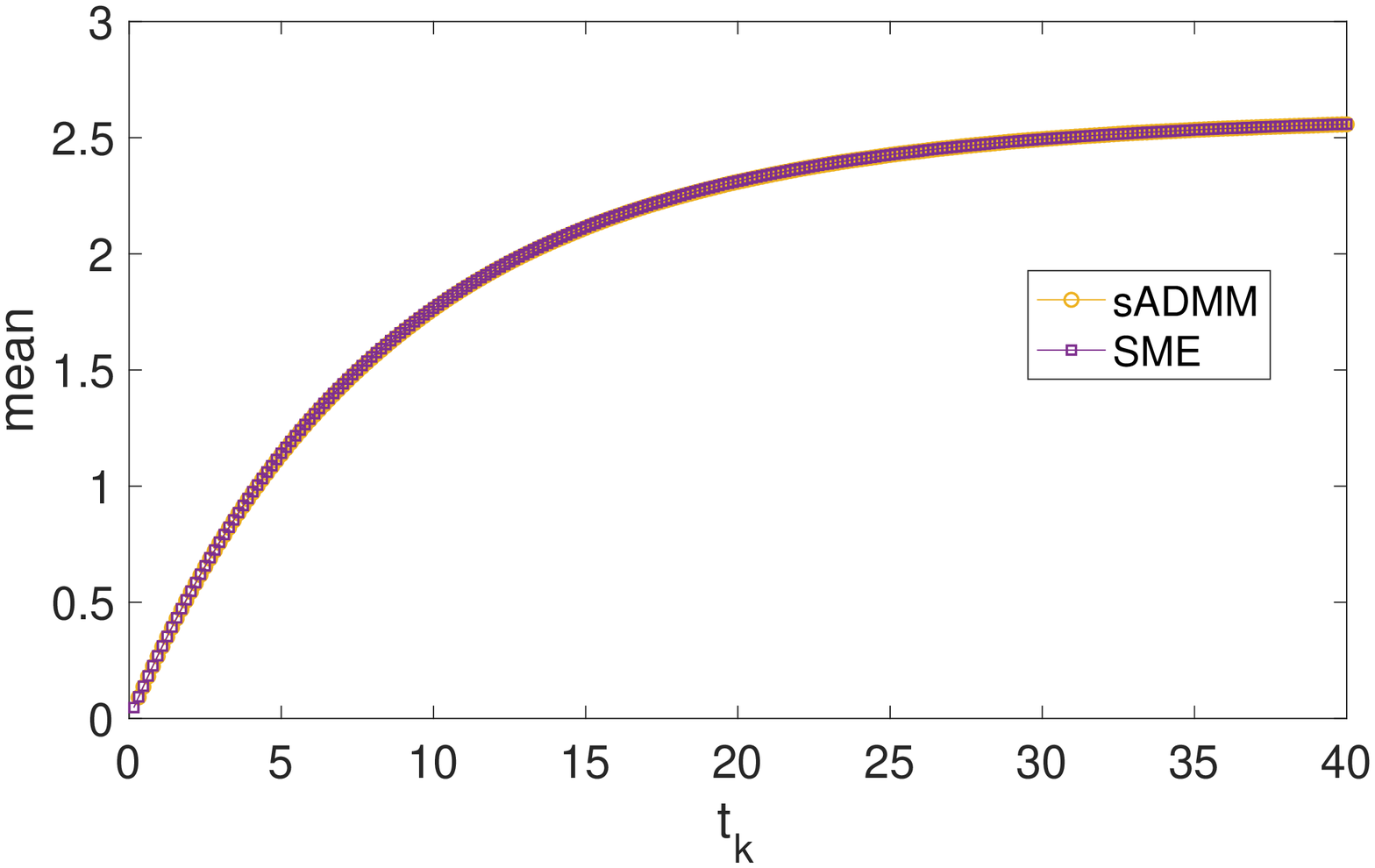}
\includegraphics[width=0.34\textwidth]{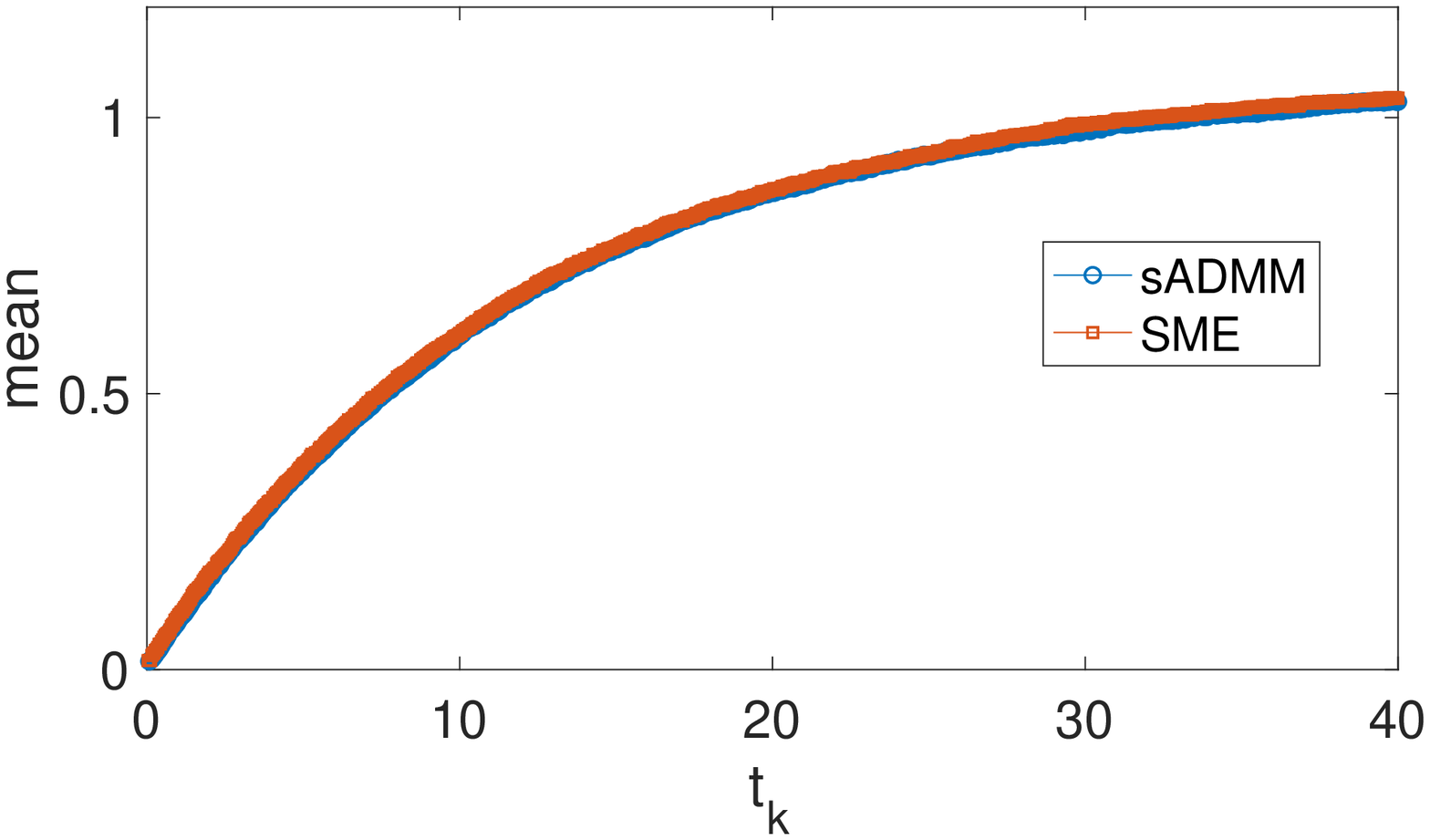}
\end{center}
\caption{The mean  of  $\varphi(x_k) $ from sADMM and $\varphi(X_{k\eps})$  from the SME.
top: $g(z)=\frac12 \beta\norm{z}^2_2$. bottom: $g(z)=\beta\norm{z}_1$.
The results are based on  100 independent runs.} 
\label{fig:reg-mean}
\end{figure}

\begin{figure}
\begin{center}
\includegraphics[width=0.34\textwidth]{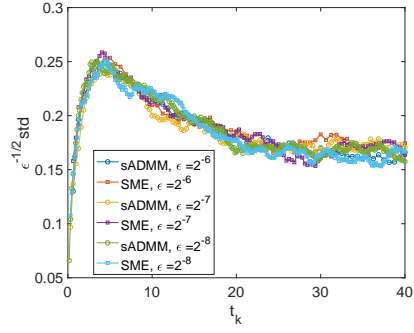}
 \end{center}
\caption{The rescaled std  of  $\varphi(x_k) $ from sADMM and $\varphi(X_{k\eps})$  from the SME.
  $g(z)=\frac12 \beta\norm{z}^2_2$.  
The results are based on  400 independent runs.} 
\label{fig:ridge-std}
\end{figure}
%

\section{Conclusion}
 In this paper, we have use the stochastic modified equation(SME) to analyze the dynamics of stochastic ADMM in the   large $\rho$ limit (i.e., small 
 step-size $\eps$ limit). It is a first order 
 weak approximation to a general family of stochastic ADMM algorithms,
 including the standard, linearized and gradient-based ADMM with relaxation $\alpha\neq 1$.
 
 Our new continuous-time analysis is the first analysis of stochastic version of ADMM. 
It faithfully captures
the fluctuation of the stochastic ADMM solution and provides  a mathematical clear and insightful way to  understand the dynamics of stochastic ADMM algorithms.
 
 It is a substantial complementary to the existing 
ODE-based continuous-time analysis \cite{francca2018admm,pmlr-v97-yuan19c} for the deterministic ADMM. It is also an important mile-stone for understanding continuous time limit of stochastic algorithms other than stochastic gradient descent (SGD), as we observed new phenonmons like the joint fluctuation of $x$, $z$ and $r$.
 We provide solid numerical experiments verifying our theory on several examples, including smooth function like quadratic functions and non-smooth function like $\ell_1$ norm.

\section{Future Work}
There are a few natural directions to further explore in future. 

First, in the theoretic analysis aspect, for simplicity of analysis, we derive our mathematical proof based on smoothness of $f$ and $g$. As we observed empirically, for non-smooth function like $\ell_1$ norm, our continuous-time limit framework would derive a stochastic differential inclusion. A natural follow-up of this work would be develop formal mathematical tools of stochastic differential inclusion to extend our proof to  non-smooth functions.

Second, from our stochastic differential equation, we could develop practical rules to choose adaptive step-size $\epsilon$ and batch size by  precisely computing the optimal diffusion-fluctuation trade-off to 
accelerate  convergence of stochastic ADMM. 
\bibliographystyle{icml2020}


\onecolumn
\pagebreak
\appendix

\bigskip
 
\begin{center}
{\Large{\bf Appendix: Stochastic Modified Equations for Continuous Limit of Stochastic ADMM}
}
\end{center}

\bigskip
 
 \section{Weak Approximation and Stochastic Modified Equations}
\label{app:A}
We introduce and review the concepts  for the weak  approximation and the stochastic modified equation.
\begin{definition}[weak convergence]
We say the family (parametrized by $\eps$) of the stochastic sequence $\{x^\eps_k: k\geq 1 \}$, $\eps>0$, weakly converges to
(or is a weak approximation to),  
a family of   continuous-time Ito processes  $\{X^\eps_t: t\in \RR^+\}$ with the  order $p$
if they satisfy the following conditions:
For any time interval $T>0$ and for any test function $\varphi$ 
such that $\varphi$ and its partial derivatives up to order $2p+2$ belong to 
$\mathcal{F}$,
there exists a constant $C>0$ and $\eps_0>0$ such that  for any $\eps < \eps_0$,
\begin{equation}
\max_{1\leq k\leq  \lfloor
T/\eps\rfloor  
}\left | \e \varphi(X^\eps_{k\eps})-\e \varphi(x^\eps_{k})
\right| \leq C \eps^p,
 \end{equation}
\end{definition}
The constant $C$ in the above inequality  and $\eps_0$, independent of $\eps$, may depend on $T$ and $\varphi$.
For the conventional applications to  numerical method for SDE
\cite{Milstein1995book}, $X^\eps$ may not depended on $\eps$;
for the stochastic modified equation in our problem, $X^\eps$ does depend on $\eps$. 
We drop  the subscript $\eps$  in $ x^\eps_k $ and  $X^\eps_t$ 
for notational ease whenever there is no ambiguity.

The idea of using the weak approximation and the stochastic modified equation 
was originally proposed by  \cite{LTE2017SME}, which
is based on an important  theorem    due to  \cite{Milstein1986}. In brief, this Milstein's theorem
links the one step difference, which has been detailed above,
to the global approximation in weak sense, by checking three  conditions on 
the momentums of one step difference. 
Since we only consider  the first order weak approximation, 
the Milstein's theorem is  introduced
in a simplified form below for only $p=1$.
The more general situations can be found in Theorem 5 in \cite{Milstein1986},
Theorem 9.1 in \cite{Milstein1995book} and  Theorem 14.5.2 in \cite{KPSDEbook2011}.

Let  the stochastic sequence $\{x_k\}$ be  recursively defined by the  iteration  
 written in the form associated with a function $\mathcal{A}(\cdot,\cdot, \cdot)$:
\begin{equation}\label{eq:AAite}
x_{k+1}=x_k- \eps \mathcal{A}(\eps, x_k,\xi_{k+1}), ~~k\geq 0
\end{equation}
where   $\{\xi_k: k\geq 1\}$ are iid random variables.
$x_0=x\in\RR^d$.
Define the one step difference  $\bar{\Delta} = x_1 -x$.
We use the  parenthetical subscript to   
denote the dimensional components of a vector like
 $\bar{\Delta} =(\bar{\Delta}_{(i)}, 1\leq i\leq d )$.
 
Assume that there exists a function $K_1(x)\in \mathcal{F}$ 
such that $\bar{\Delta}$ satisfies the bounds  of the fourth momentum
\begin{equation}\label{con:m4}
\abs{ \e ( \bar{\Delta}_{(i)}\bar{\Delta}_{(j)} \bar{\Delta}_{(m)}\bar{\Delta}_{(l)})  } \leq K_1(x) \eps^3
\end{equation}
  for any component indices $i,j,m,l \in \{ 1,2,\ldots, d\}$ and 
 any $x\in \RR^d$,

For any arbitrary  $\eps>0$, consider  the family of the Ito processes  $X^\eps_t$ defined by a  
stochastic differential equation whose noise depends on the parameter $\eps$,
\begin{equation} 
\label{eqn:SDE}
dX_t = b(X_t)dt + \sqrt{\eps} \sigma(X_t)dW_t,
\end{equation}
$W_t$ is the standard Wiener process in $\mathbb{R}^d$.
The initial  is $X_0=x_0=x$.
The coefficient functions $b$ and $\sigma$   satisfy certain standard conditions;
see \cite{Milstein1995book}.
Define the one step difference  $ {\Delta} = X_{\eps} -x$ for the SDE \eqref{eqn:SDE}.

\begin{theorem}[Milstein's weak convergence theorem]
\label{thm:Milstein}
If  there exist a constant $K_0$  
   and a function $K_2(x) \in \mathcal{F}$ ,   such that  the following   conditions  
   of the first three moments on the error $\Delta - \bar{\Delta} $:
 \begin{subequations}
 \label{con:mm}
\begin{align}\label{con:m1}
&\abs{ \e(X_\eps - X_1)  } \leq K_0 \eps^2
\\
\label{con:m2}
&\abs{ \e  (\Delta_{(i)}  \Delta_{(j)})  - \e (\bar{\Delta}_{(i)} \bar{\Delta}_{(j)})  }  \leq K_1(x) \eps^2
\\
\label{con:m3}
&\abs{ \e ( \Delta_{(i)} \Delta_{(j)} \Delta_{(l)})- \e ( \bar{\Delta}_{(i)} \bar{\Delta}_{(j)}\bar{ \Delta}_{(l)})  } \leq K_1(x) \eps^2
\end{align}
 \end{subequations}
 hold for any   $i,j,l\in \{ 12,\ldots, d\}$ and 
 any $x\in \RR^d$,
 then 
 $\{x_k\}$ weakly converges to $\{X_t\}$ with the order 1.
\end{theorem}
 In light of the above theorem, we will now call equation  \eqref{eqn:SDE}
  the   {\it stochastic modified equation} (SME) of the iterative scheme \eqref{eq:AAite}.

For the SDE \eqref{eqn:SDE} at the small noise $\eps$,  by the Ito-Taylor expansion,  it is well-known that $\e  {\Delta}  = b  (x) \eps +  \mathcal{O}(\eps^2)$
and $\e [  {\Delta}  {\Delta} ^\top ]=  \left(b(x)b(x)^\top + \sigma(x)\sigma(x)^\top\right)\eps^2 + \mathcal{O}(\eps^3)$
and  $   \e ( \Pi_{m=1}^s {\Delta}_{(i_m)} )   =\mathcal{O}(\eps^3)$ for all integer $s\geq 3$ and the component index $i_m=1,\ldots,d$.
Refer to \cite{KPSDEbook2011} and  Lemma 1  in \cite{LTE2017SME}.
So, the main receipt to  apply the Milstein's theorem is to examine the conditions of the momentums
for the discrete sequence $\bar{\Delta} = x_1-x_0 $.

One prominent   work   \cite{LTE2017SME} is to use   the SME 
as a weak approximation  to understand the dynamical behaviour of the stochastic gradient descent (SGD). 
The prominent advantage of this technique is that
the fluctuation in the SGD iteration can be well captured by the fluctuation in the SME.
Here is the brief result.  For the composite minimization problem $$\min_{x\in \RR} f(x)=\e_\xi f(x,\xi),$$
  the SGD iteration is $x_{k+1}=x_k - \eps f'(x_k,\xi_{k+1})$ with the step size $\eps$,
   then by Theorem \ref{thm:Milstein}, the corresponding SME  of first order approximation is
  \begin{equation}
  \label{eq:sde}
  dX_t = -  f'(x) dt + \sqrt{\eps}\sigma(x) dW_t
  \end{equation} 
with $\sigma(x)=\mbox{std}_\xi(f'(x,\xi))=(\e[ (f'(x)-f'(x,\xi))^2 ])^{1/2}$. Details can be found in  \cite{LTE2017SME}.
The SGD here is analogous to the  forward-time Euler-Maruyama approximation
since $\mathcal{A}(\eps,x,\xi)=f'(x,\xi)$.

\section{Proof of main theorems}
\label{sec:pfthm}
\label{sec2}
The one step difference is important to consider the weak convergence of the discrete scheme 
\eqref{GSADMM}.  The  question is that for one single iteration, from step $k$ to  step $k+1$, what is the order of  the {\it change}  of the states $(x,z,u)$. Since For notational ease, we   drop the random variable 
$\xi_{k+1}$  in the scheme \eqref{GSADMM}; the readers bear in mind that $f$ and its derivatives involve $\xi$.

We work on the general ADMM scheme \eqref{GSADMM}.
 The optimality conditions for the scheme \eqref{GSADMM} are
\begin{subequations} \label{GADMM}
\begin{align}
& \omega_1 \eps   f'(x_k) +
 (1-\omega_1)\eps f'(x_{k+1})
 +\eps A^\top \lambda_k
\notag
 \\
&\qquad +    A^\top
 \left ( \omega  Ax_{k} + (1-\omega )A x_{k+1}  - z_k  \right)
   + c(x_{k+1}-x_k) =0
   \label{eq:opt-x--update}
\\
 \label{eq:opt-z--update}
 &
 \eps g'(z_{k+1}) =
  \eps \lambda_k +     \alpha A x_{k+1} +(1-\alpha)z_k
- z_{k+1}
 \\
 &
 \eps \lambda_{k+1}= \eps \lambda_{k} +   \alpha    A  x_{k+1}+(1- \alpha) z_k{} -  z_{k+1}
  \label{eq:opt-lambda--update}
 \end{align}
 \end{subequations}

%
Note that due to \eqref{eq:opt-z--update} and \eqref{eq:opt-lambda--update}, the last condition \eqref{eq:opt-lambda--update} can be   
replaced 
by $
\lambda_{k+1} = g'(z_{k+1}).$
So, without loss of generality, one can assume that 
\begin{equation}
\lambda_{k'} \equiv g'(z_{k'})\end{equation}
 for any integer $k'\geq 1$.
The optimality conditions \eqref{GADMM} now can be written only 
in the variables $(x,z)$:
\begin{subequations} \label{GADMM2}
\begin{align}
& \omega_1 \eps   f'(x_k) +
 (1-\omega_1)\eps f'(x_{k+1})
 +\eps A^\top g'(y_k)
\notag
 \\
& \qquad +    A^\top
 \left ( \omega  Ax_{k} + (1-\omega )A x_{k+1}  - z_k  \right)
   + c(x_{k+1}-x_k) =0
   \label{eq:opt-x-update}
\\
 \label{eq:opt-z-update}
 &
 \eps g'(z_{k+1}) - \eps g'(z_k) =    \alpha A x_{k+1} +(1-\alpha)z_k
- z_{k+1}
  \end{align}
 \end{subequations}
 
As $\eps\to 0$, we seek the asymptotic expansion of $x_{k+1}-x_{k} $
from \eqref{eq:opt-x-update} and the asymptotic expansion of $z_{k+1}-z_{k} $
from  \eqref{eq:opt-z-update}.
The first result is that
\begin{subequations}
\label{eq:xzO1}
\begin{align}
 x_{k+1}-x_{k} &= -M^{-1}  A^\top r_k + c_k \eps,
\label{eq:xO1}
\\ 
 z_{k+1}-z_{k}& = \alpha ( I -   A M^{-1}  A^\top ) r_k +  c'_k \eps,
\label{eq:yO1}
\end{align}
 \end{subequations}
where $r_k$ is the  {\it residual}
\begin{equation}
r_k:= Ax_k-z_k
\end{equation}
and 
the matrix $M$ is \begin{equation}
\label{def:M}
M=M_{c,\omega}:=c+(1-\omega)A^\top A.
\end{equation}
The constant $c_k$ and $c'_k$ are independent of $\eps$ but related to $f'$, $g'$ and other parameters $\alpha,\omega,\omega_1$.
Throughout the rest of the paper,  we shall use the notation $\mathcal{O}(\eps^p)$ to denote the terms $c_k \eps^p$, for $p=1,2,\ldots$.
Given any input $(x_{k},z_{k})$,  since $r_k=Ax_k-z_k$ may not be zero,
then as the step size $\eps \to 0$,  \eqref{eq:xO1} and \eqref{eq:xO1} show that   $(x_{k+1},z_{k+1})$ does not 
converge to $(x_{k},z_{k})$.
However  we can  show that the residual after one step iteration 
$r_{k+1}$ is always a small number on the order $\mathcal{O}(\eps)$, so that the consist
condition that as $\eps\to0$, $(x_{k+1}, z_{k+1})$ tends to $(x_k,z_k)$ holds.

\begin{proposition} 
\label{p:rpp}
We have the following property for the propagation of the residual: 
\begin{equation}
\label{eq:r-update}
  r_{k+1}
= \left( 1-\alpha \right) ( I -   A M^{-1}  A^\top ) r_k  + \mathcal{O}(\eps).
   \end{equation}
\end{proposition}
\begin{proof}
By using \eqref{eq:opt-z-update}  
 and \eqref{eq:yO1}, 
 \begin{equation*}
 \begin{split}
 r_{k+1}
 &= Ax_{k+1} - z_{k+1}
   = \left( \frac{1}{\alpha} -1 \right)(z_{k+1}-z_k)+ 
   \frac{\eps}{\alpha}( g'(z_{k+1} ) - g'(z_k))
  \\
  &= \left( 1-\alpha \right) ( I -   A M^{-1}  A^\top ) r_k  + \mathcal{O}(\eps).
  \end{split}
   \end{equation*}
 \end{proof}
  \begin{remark}
     
  If $\alpha=1$, 
 the leading term $\left( 1-\alpha \right) ( I -   A M^{-1}  A^\top ) $ vanishes.
 There are some special cases where the matrix $I -   A M^{-1}  A^\top$ is zero: 
 (1)
     $A$ is an invertible square matrix 
 and $M=M_{0,1}=A^\top A$.
 (2) $A$ is an orthogonal matrix ($AA^\top=A^\top A=I$) and the constants satisfy $\omega=c$ such that that
 $M=I$. 
 \end{remark}

The above proposition is for an arbitrary residual $r_k$ as the input in   one step iteration.
If we  choose $r_0=0$ at the initial step by setting $z_0=Ax_0$, then Proposition \ref{p:rpp}
 shows that $r_1=Ax_1-y_1$ become $\mathcal{O}(\eps)$ after one iteration.
In fact,   with assumption $\alpha=1$, we can show $r_{k'}, \forall k'\geq 0$, can be reduced to the order $\eps^2$
 by mathematical induction.

\begin{proposition}
\label{p:rp2}
If $r_k = \mathcal{O}(\eps)$, then 
\begin{equation}
\label{eq:r-update2}
r_{k+1} =(1-\alpha + \eps  \alpha  g''(z_k)) (r_k + A (x_{k+1}-x_k))  + \mathcal{O}(\eps^3).
\end{equation}
If $\alpha=1$, equation \eqref{eq:r-update2} reduces to the second order smallness:
\begin{equation}
\label{eq:r-update22}
r_{k+1} = \eps  \alpha  g''(z_k) (r_k + A (x_{k+1}-x_k))  + \mathcal{O}(\eps^3) =  \mathcal{O}(\eps^2) .
\end{equation}

\end{proposition}
\begin{proof}
 Since that $r_k=Ax_k- z_k=\mathcal{O}(\eps)$, then 
 the one step difference $x_{k+1}-x_k$ and $z_{k+1}-z_k$  are both at order $\mathcal{O}(\eps)$ because of 
\eqref{eq:xO1} and \eqref{eq:yO1}.
We  solve $\delta z:=z_{k+1}-z_k$ from  \eqref{eq:opt-z-update} 
by linearizing the implicit term $g'(z_{k+1})$
with the assumption that the third order derivative of $g$ exits:
$$ \eps g''(z_k)\delta z + \eps \mathcal{O}( (\delta z)^2 ) + \delta z= \alpha (r_k+ A\delta x). $$ 
where $\delta x:= x_{k+1}-x_k$. 
Then since $\mathcal{O}((\delta z)^2)=\mathcal{O}(\eps^2)$, 
 the expansion of $\delta z=z_{k+1}-z_k$ in $\eps$ is 
 \begin{equation}
\label{exp:delta-y}
z_{k+1}-z_k = \delta z = \alpha (1-\eps g''(z_k))  (r_k + A\delta x)     + \mathcal{O}(\eps^3)
\end{equation}
Then
\begin{equation*}
 \begin{split}
r_{k+1}
&=r_k +  A  (x_{k+1}-x_k) - (z_{k+1}-z_k)
\\
&=\Big(1-\alpha + \eps  \alpha  g''(z_k) \Big)(r_k + A (x_{k+1}-x_k))  + \mathcal{O}(\eps^3)
\\
&=(1-\alpha) ( r_k + (x_{k+1}-x_k)) +  \eps  \alpha  g''(z_k)(r_k + A (x_{k+1}-x_k)) + \mathcal{O}(\eps^3)
\end{split}
\end{equation*}
\end{proof}

\begin{remark}
\eqref{eq:r-update2} suggests that $r_{k+1} = (1-\alpha)r_{k} + \mathcal{O}(\eps)$.
So the condition for the convergence $r_k \to 0$ as $k\to \infty$  is 
$\abs{1-\alpha}<1$, which matches the range  $\alpha\in (0,2)$ used in the relaxation scheme. 
\end{remark}

Now with the assumption $y_0=Ax_0$ at initial time,
the above analysis shows that $r_k$ is  $\mathcal{O}(\eps)$ and 
the one step difference $x_{k+1}-x_k$ and $z_{k+1}-z_k$ are on the order $\mathcal{O}(\eps)$ by \eqref{eq:xzO1}.
We shall pursue a more accurate expansion of the one step difference $x_{k+1}-x_k$  than  \eqref{eq:xzO1}.
Write $f'(x_{k+1})=f'(x_k)+ f''(x_k)(x_{k+1}-x_{k}) + \mathcal{O}( (x_{k+1}-x_{k})^2)$ in   equations \eqref{GADMM2}.
 The asymptotic analysis shows the result below.
\begin{proposition}
\label{p:xp3}
 As $\eps\to 0$,   the expansion of the one step difference 
 $x_{k+1}-x_k$ is 
 \begin{equation}
\label{exp:x}
\begin{split}
M (x_{k+1}-x_{k}) &=-  A^\top r_k
 - \eps  
\left(
f'(x_k)+A^\top g'(y_k) 
\right)
\\
&+ \eps ^2  
(1-\omega_1)
f''(x_k) M^{-1} 
\bigg(
f'(x_k) + A^\top g'(y_k) + \frac{1}{\eps} A^\top r_k
\bigg)
 + \mathcal{O}(\eps^3).
 \end{split}
\end{equation}
\end{proposition}
This expression does not contain the parameter $\alpha$ explicitly, but the residual 
$r_k=Ax_k-y_k$ significantly depends on $\alpha$ (see Proposition \ref{p:rp2}).
If $\alpha=1$, then  $r_k$ is on the order of $\eps^2$, which hints there is no contribution
from $r_k$ toward the weak approximation of $x_k$  at the order 1.
But for the relaxation case where $\alpha\neq 1$, $r_k$ contains the first order term 
coming from $z_{k+1}-z_k$.

  To obtain a second order smallness for some  ``residual'' for the relaxes scheme 
 where $\alpha\neq 1$, we need a new definition, $\alpha$-residual, to account for the 
 gap induced by $\alpha$. Motivated by \eqref{eq:opt-z--update},  we first define
 \begin{equation}
 \label{def:alpha-residual}
 r^{\alpha}_{k+1}:=\alpha Ax_{k+1}+(1-\alpha)z_k -z_{k+1}.
 \end{equation}
 It is connected to the original residual $r_{k+1}$ and $r_{k}$ since it is easy to check that 
  \begin{equation}
  \label{eqn:a-residual}
  \begin{split}
  r^\alpha _{k+1}&
  = \alpha r_{k+1}+(\alpha-1) (z_{k+1}-z_k)
    =\alpha r_{k}+ \alpha   A (x_{k+1}-x_k )-(z_{k+1}-z_k)
  \end{split}
  \end{equation}
   But $r^\alpha_{k+1}$ in fact involves information at two successive steps.
 Obviously, when $\alpha=1$, this $\alpha$-residual $r^\alpha$ is the original residual $r=Ax-y$.
 In our proof, we need a modified $\alpha$-residual,   denoted by   
 \begin{equation}
 \label{def:ar}
 \widehat{r}^\alpha_{k+1}:=\alpha r_{k}+(\alpha-1) (z_{k+1}-z_k)
 \end{equation}
We can show that both $r^{\alpha}_{k+1}$ and $ \widehat{r}^\alpha_{k+1}$ 
are as small as $\mathcal{O}(\eps^2)$ as $\eps$ tends to zero.

 \begin{proposition}
 \label{p:ar}
  $r^{\alpha}_{k+1}=\mathcal{O}(\eps^2)$ and $\widehat{r}^{\alpha}_{k+1} = \mathcal{O}(\eps^2)$.
  \end{proposition}

 \begin{proof}
  In fact, \eqref{exp:delta-y} is 
$z_{k+1}-z_k =\alpha(1-\eps g''(z_k))   (r_k + A(x_{k+1}-x_k)) +\mathcal{O}(\eps^3)$.
 By the second equality  of  \eqref{eqn:a-residual},
\eqref{exp:delta-y} becomes
$ z_{k+1}-z_k = (1-\eps g''(z_k))    ( r^\alpha_{k+1}+z_{k+1}-z_k) +\mathcal{O}(\eps^3)
 $, i.e.,
 \[
 \begin{split}
      r^\alpha_{k+1}
      &= \eps (1+\eps g''(z_k)) g''(z_k) (z_{k+1}-z_k) +\mathcal{O}(\eps^3)\\
            &= \eps   g''(z_k) (z_{k+1}-z_k) +\mathcal{O}(\eps^3)
\end{split}
\]
which is $\mathcal{O}(\eps^2)$ since  $z_{k+1}-z_k=\mathcal{O}(\eps)$.

The difference between $ (z_{k+1}-z_k)$
 and $ (z_{k+2}-z_{k+1})$, is  at the order $\eps^2$ due to truncation error of the central difference scheme,
 Then we have the conclusion 
 $ \alpha r_{k+1}+(\alpha-1) (z_{k+2}-z_{k+1})$,
 i.e, 
 \begin{equation}
 \label{eqn:smallaresdual}
 \widehat{r}^\alpha_{k+1}= \alpha r_{k}+(\alpha-1) (z_{k+1}-z_{k}) = \mathcal{O}(\eps^2)
\end{equation} 
by shifting the subscript $k$ by one.
 
\end{proof}

%

\begin{corollary}
\label{cor:rk}
 \begin{equation}
 r_k = \left( \frac{1}{\alpha}-1\right) (z_{k+1}-z_k)+ \mathcal{O}(\eps^2)
=\left( \frac{1}{\alpha}-1\right) A(x_{k+1}-x_k) + \mathcal{O}(\eps^2)
 \end{equation}
 and it follows  $z_{k+1}-z_k=  A(x_{k+1}-x_k) + \mathcal{O}(\eps^2)$.
\end{corollary}

\begin{proof}
By \eqref{def:ar} and the above proposition, 
 we have $  r_{k}=(\frac{1}{\alpha}-1) (z_{k+1}-z_{k}) + \mathcal{O}(\eps^2)$.
Furthermore, due to  \eqref{exp:delta-y},  
 $  r_{k}=(\frac{1}{\alpha}-1) (z_{k+1}-z_{k}) + \mathcal{O}(\eps^2)
 =(1-{\alpha} )  (r_k + A (x_{k+1}-x_k)) + \mathcal{O}(\eps^2) $
 which gives 
 \begin{equation*}
 r_k = \left( \frac{1}{\alpha}-1\right) A(x_{k+1}-x_k) + \mathcal{O}(\eps^2)
 \end{equation*}

\end{proof}

\begin{proof}[Proof of Theorem \ref{Gthm}]
Combining Proposition \ref{p:xp3}  and Corollary \ref{cor:rk}, and noting 
 the Taylor expansion of $g'(z_k)$:
 $g'(y_k) = g'(Ax_k-r_k)=g'(Ax_k)  +  \mathcal{O}(\eps)
 $ since $r_k= \mathcal{O}(\eps)$
 and putting back random $\xi$ into $f'$,  
 we have 
\begin{equation}
\begin{split}
M (x_{k+1}-x_{k})=
&-\eps   
 \left (  
 f'(x_k,\xi_{k+1})+A^\top g'(Ax_k)  \right)
 \\
 &- \left(\frac{1}{\alpha}-1\right) A^\top A  (x_{k+1}-x_{k})
+
 \mathcal{O}(\eps^2) 
\end{split}
\end{equation}

For convenience, introduce the matrix 
\begin{equation}
\widehat{M} := M + \frac{1-\alpha}{\alpha}A^\top A
=c+\left(\frac{1}{\alpha}-\omega\right)A^\top A.
\end{equation}
and let 
$$\widehat{x}_k:= \widehat{M}x_k,
~~\text{ and }~~\delta \widehat{x}_{k+1} = \widehat{M}(x_{k+1}-x_k)$$
Then
$$
\delta \widehat{x} =-\eps  
V'(x,\xi)  +
\eps^2      \left( 
(1-\omega_1)f''M^{-1} V'(x)- A^\top \theta
  \right) 
+
 \mathcal{O}(\eps^3)
$$

The final step is to compute the momentums  in the Milstein's theorem  Theorem \ref{thm:Milstein} as follows
\begin{enumerate}[(i)]
\item
\begin{equation}
\label{eq:EEdx}
\e[\delta \widehat{x}] = -\eps \e V'(x,\xi)+\mathcal{O}(\eps^2) = -\eps   V'(x)+\mathcal{O}(\eps^2) 
\end{equation}
\item 
\begin{align*}
\e[\delta \widehat{x}\, \delta \widehat{x}^\top ]
&=\eps^2 
\e\left( \left[ f'(x,\xi) + A^\top g'(x)  
\right]\left[ f'(x,\xi)^\top +   g'(x)^\top A
\right]\right) + \mathcal{O}(\eps^3)
\\
 &=\eps^2 
\left(
V'(x) V'(x)^\top
\right) - \eps^2 \left(f'(x) + A^\top g'(x) )(f'(x)^\top +   g'(x)^\top A)\right) 
\\
&\qquad + \eps^2
\e\left( \left[ f'(x,\xi) + A^\top g'(x)  
\right]\left[ f'(x,\xi)^\top +   g'(x)^\top A
\right]\right)+ \mathcal{O}(\eps^3)
\\
&= \eps^2 
\left(
V'(x) V'(x)^\top
\right) + \eps^2
\e
\left[
\left ( f'(x,\xi)-f'(x)  \right)\left ( f'(x,\xi)-f'(x)  \right)^\top 
\right] +\mathcal{O}(\eps^3)
\end{align*}

\item 
It is trivial that 
$\e[\Pi_{j=1}^s \delta x_{i_j}] = \mathcal{O}(\eps^3)$ for $s\geq 3$ and $i_j=1,\ldots, d$.
\end{enumerate}
So, Theorem \ref{Gthm} is proved.
\end{proof}

\begin{proof}[Proof of Theorem \ref{Sthm}]
Theorem \ref{Sthm} is a special case of Theorem \ref{Gthm}.
Let $\alpha=1$, $\omega=0$, $c=0$,
 then  $\widehat{M}=A^\top A$.
\end{proof}

\newpage

\end{document}